\documentclass[11pt,reqno]{amsart}
\usepackage{enumerate}
\usepackage{amsrefs}
\usepackage{amsfonts, amsmath, amssymb, amscd, amsthm, bm, cancel,mathrsfs}
\usepackage{url}
\usepackage{graphicx}
\usepackage[bookmarks=false]{hyperref}
\usepackage{multicol}
\usepackage{comment}
\usepackage{tikz}
\usepackage[margin=1.1in]{geometry}
\newtheorem{prop}{Proposition}[section]
\newtheorem{thm}[prop]{Theorem}
\newtheorem*{thmA}{Theorem A}
\newtheorem*{thmB}{Theorem B}
\newtheorem{lem}[prop]{Lemma}
\newtheorem{cor}[prop]{Corollary}
\newtheorem{defn}[prop]{Definition}

\theoremstyle{definition}
\newtheorem*{ack}{Acknowledgments}
\theoremstyle{remark}
\newtheorem{rem}[prop]{Remark}
\numberwithin{equation}{section}
\allowdisplaybreaks

\begin{document}
	\begin{abstract}
		We consider a locally constrained curvature flow in a static rotationally symmetric space $\mathbf{N}^{n+1}$, which was firstly introduced by Hu and Li\cite{HL-2019} in the hyperbolic space. We prove that if the initial hypersurface is graphical, then the smooth solution of the flow remains to be graphical, exists for all positive time $t\in[0,\infty)$ and converges to a slice of $\mathbf{N}^{n+1}$ exponentially in the smooth topology. Moreover, we prove that the flow preserves static convexity if the initial hypersurface is close to a slice of $\mathbf{N}^{n+1}$ in the $C^1$ sense. As applications, we prove a family of weighted geometric inequalities for static convex domains which is close to a slice of $\mathbf{N}^{n+1}$ in the $C^1$ sense. 
	\end{abstract}

\title[The weighted geometric inequalities for static convex domains]{The weighted geometric inequalities for static convex domains in static rotationally symmetric spaces}

\author[S. Pan]{Shujing Pan}
\address{School of Mathematical Sciences, University of Science and Technology of China, Hefei 230026, P.R. China}
\email{\href{mailto:psj@ustc.edu.cn}{psj@ustc.edu.cn}}
\author[B. Yang]{Bo Yang}
\address{Department of Mathematical Sciences, Tsinghua University, Beijing 100084, P.R. China}
\email{\href{mailto:ybo@tsinghua.edu.cn}{ybo@tsinghua.edu.cn}}

\maketitle
\tableofcontents
\section{Introduction}\label{sec-intro}

 Let $X_0:\mathbb{S}^n\to \mathbf{N}^{n+1}$ be a smooth embedding such that $M_0=X_0(\mathbb{S}^n)$ is a hypersurface in $\mathbf{N}^{n+1}$, where $\mathbf{N}^{n+1}$ is a rotationally symmetric space equipped with the metric \eqref{eq-metric}. We consider the smooth family of embeddings $X:\mathbb{S}^n\times [0,T)\to \mathbf{N}^{n+1}$ satisfying
\begin{equation}\label{flow-WVPF}
	\left\{\begin{aligned}
		\frac{\partial}{\partial t}X(x,t)=&~(n-\frac{uH}{\phi'})\nu(x,t),\\
		X(\cdot,0)=&~X_0(\cdot),
	\end{aligned}\right.
\end{equation}
where $\nu$ is the unit outer normal of $M_t=X(\mathbb{S}^n,t)$, $H$ is the mean curvature of $M_t$ and $u$ is the support function of $M_t$, which is defined as $u=\langle{\phi\partial_{r},\nu}\rangle$.

\begin{defn}\label{defn-N}
Let $\mathbf{N}^{n+1}=[r_0,\bar{r})\times\mathbb{S}^n$ be a rotationally symmetric space equipped with the following warped product metric:
\begin{equation}\label{eq-metric}
	\bar{g}=dr^2+\phi^2(r)\sigma,\quad r\in[r_0,\bar{r}),
\end{equation}
where $r_0\geq 0$ and $\bar{r}$ is allowed to be $+\infty$. $\sigma$ is the round metric on $\mathbb{S}^n$, the functions $\phi,\phi',\phi''>0$ and satisfy the following inequalities:
	\begin{equation}\label{in-phiphi}
	0\leq (\phi')^2-\phi\phi''\leq 1.
\end{equation}
\end{defn}

\begin{rem}
   In this paper, most results apply to all dimensions $n\geq 1$, except for Corollary \ref{Cor-wii} and Corollary \ref{Cor-Sci}, where we need extra conditions $n\geq 2$ or $n\geq 3$.
\end{rem}
\begin{rem}
  In the remaining of this paper, we always assume that the rotationally symmetric space $\mathbf{N}^{n+1}$ equipped with the metric \eqref{eq-metric} has the property that $\phi,\phi'>0$. We only emphasize the assumptions on $\mathbf{N}^{n+1}$ in addition to the property that $\phi,\phi'>0$.
\end{rem}

Let $S(r)=\{r\}\times\mathbb{S}^n$ be the slice of $\mathbf{N}^{n+1}$ with radius $r$ and $B(r):=[r_0,r]\times \mathbb{S}^n$ be the bounded domain enclosed by $S(r)$ and $S(r_0)$. A hypersurface $M$ of $\mathbf{N}^{n+1}$ is called graphical, if it can be expressed as a graph of a smooth and positive function $r(\theta)$ on $\mathbb{S}^n$:
\begin{equation*}
	M=\{(r(\theta),\theta)\in[r_0,\bar{r})\times\mathbb{S}^n\}.
\end{equation*}

Assume that $\Omega$ is a bounded domain in $\mathbf{N}^{n+1}$ which is enclosed by a hypersurface $M$ and $S(r_0)$. We define the weighted $\alpha$-volume of $\Omega$, the weighted area and the weighted mean curvature integral of $M$ as follows:
\begin{align}
		V_{\phi}^{\alpha}(\Omega)&=\int_{\Omega}{(\phi')^{\alpha}}dv,\label{eq-wevo}\\
		A_{0,\phi}(M)&=\int_{M}{\phi'}d\mu,\label{eq-wear}\\
           A_{1,\phi}(M)&=\int_{M}{\phi'H}d\mu.\label{eq-weH}
	\end{align}
Note that $V_{\phi}^0(\Omega):=V(\Omega)$, is the volume of the domain $\Omega$ and for $\alpha=1$, we denote $V_{\phi}^1(\Omega)$ as $V_{\phi}(\Omega)$ and call it the weighted volume of $\Omega$ for simplicity.

 By Lemma \ref{lem-mono}, we know that the flow \eqref{flow-WVPF} is a weighted volume preserving flow, i.e., $V_{\phi}(\Omega_t)\equiv V_{\phi}(\Omega_0)$ for $t>0$ as long as the flow exists, where $\Omega_t$ is the domain encloesd by the hypersurface $M_t$ and $S(r_0)$.
 
 Denote the weighted $\alpha$-volume of $B(r)$ by $V_{\phi}^{\alpha}(r)$, which are all strictly increasing functions of $r$. Consider the single variable functions $\xi_{\alpha}(x)$ which satisfy
\begin{equation}\label{eq-xi}
	V_{\phi}^{\alpha}(r)=\xi_{\alpha}(V_{\phi}(r))
\end{equation} 
for any $r\in[r_0,\bar{r})$, then the functions $\xi_{\alpha}(x)$ are all well-defined. Similarily, there exists a family of well-defined single variable functions $\chi_{i,\alpha}(x)$ ($i=0,1$) which satisfy
\begin{equation}\label{eq-chi}
A_{i,\phi}(r)=\chi_{i,\alpha}(V_{\phi}^{\alpha}(r))
\end{equation}
for any $r\in[r_0,\bar{r})$, where $A_{0,\phi}(r)$ is the weighted area of $S(r)$ and $A_{1,\phi}(r)$ is the weighted mean curvature integral of $S(r)$.

As the first result of this paper, we prove the following convergence result for the flow \eqref{flow-WVPF} .

\begin{thm}\label{Thm-main}
	Let $\mathbf{N}^{n+1}$ be a rotationally symmetric space equipped with the metric \eqref{eq-metric} and the function $\phi$ satisfies $(\phi')^2-\phi\phi''\geq 0$. Let $X_0:\mathbb{S}^n\to \mathbf{N}^{n+1}$ be a smooth embedding such that $M_0=X_0(\mathbb{S}^n)$ is a graphical hypersurface and $\Omega_0$ is the domain encloesd by $M_0$ and the slice $S(r_0)$. Then the weighted volume preserving flow \eqref{flow-WVPF} has a unique smooth graphical solution $M_t$ for all time $t\in[0,\infty)$, the solution $M_t$ converges exponentially to a slice of $\mathbf{N}^{n+1}$ with radius $r_{\infty}$ in the smooth topology, and $r_{\infty}$ is determined by the weighted volume $V_{\phi}(\Omega_0)$.
\end{thm}

\begin{rem}
	The lower bound condition $(\phi')^2-\phi\phi''\geq 0$ is needed for obtaining the $C^1$ estimate along the flow \eqref{flow-WVPF}.  
\end{rem}

As shown in Lemma \ref{lem-mono}, we find that the weighted $\alpha$-volume $V_{\phi}^{\alpha}(\Omega)$ is monotone along the flow \eqref{flow-WVPF} under adequate conditions on $\phi''$ and $\alpha$, then the smooth convergence proved in Theorem \ref{Thm-main} yields the following geometric inequalities between the weighted volume and the weighted $\alpha$-volume of a domian $\Omega$.

\begin{cor}\label{Cor-VVphi}
    Let $\mathbf{N}^{n+1}$ be a rotationally symmetric space equipped with the metric \eqref{eq-metric} and the warped function $\phi$ satisfies $(\phi')^2-\phi\phi''\geq 0$. Assume that $\Omega\subset \mathbf{N}^{n+1}$ is a bounded domain encloesd by $S(r_0)$ and a graphical hypersurface $M$. Then we have:
    \begin{align}
        V_{\phi}^{\alpha}({\Omega})&\leq \xi_{\alpha}(V_{\phi}(\Omega)), \quad \text{if} \quad \phi''(\alpha-1)<0,\label{In-VVphi}\\
        V_{\phi}^{\alpha}({\Omega})&\geq \xi_{\alpha}(V_{\phi}(\Omega)), \quad \text{if} \quad \phi''(\alpha-1)>0,\label{In-VVphi2}
    \end{align} 
    where $\xi_{\alpha}$ is the function defined in equation \eqref{eq-xi}. The equality in \eqref{In-VVphi} or \eqref{In-VVphi2} holds if and only if $\Omega=B(r):=[r_0,r]\times\mathbb{S}^n$ for some $r$.
\end{cor}

The mean curvature type flow 
\begin{equation}\label{Flow-mct}
    \frac{\partial}{\partial t}X(x,t)=(n\phi'-uH)\nu(x,t)
\end{equation}
was firstly introduced  by Guan and Li\cite{GL-2015} in space forms, they proved that for any smooth, compact, star-shaped initial hypersurface, the solution of the flow \eqref{Flow-mct} converges exponentially to a geodesic sphere in the smooth topology. The flow has a nice feature that along the flow \eqref{Flow-mct}, the volume of the enclosed domain is constant and the surface area is monotonically decreasing, which yields that the flow hypersurface $M_t$ converges to a solution of the isoperimetric problem in space forms. Together with Wang\cite{GLW-2019}, they generalized this result to warped product spaces under assumptions on the warping function $\phi$ and Ricci curvature of the base manifold. Using the flow \eqref{Flow-mct}, the first author and Li\cite{LP-2023} solved the isoperimetric problem in the general Riemannian manifold which admits a non-trivial conformal vector field in a recent work. The flow \eqref{Flow-mct} was also generalized to the fully-nonlinear version by Guan and Li\cite{GL-2018} in the Euclidean space and by Chen, Guan, Li and Scheuer\cite{CGLS-2022} in the sphere, by replacing the mean curvature $H$ to  general curvature functions $\sigma_{k+1}/\sigma_{k},k=1,\dots,n$. There are many other interesting results concerning the locally constrained curvature flows, such as \cites{BGL-note,CS-2022,HLW-2022,KWWW-2022,SX-2019}.


Locally constrained curvature flows are considered mainly in space forms in the literature. The reason is that space forms have constant sectional curvature, so it is easy to deal with the terms involving Riemannian curvature tensors and their derivatives. Meanwhile, the quermassintegrals in space forms are well-defined and have a nice varaitional property. Then it is direct to introduce new locally constrained flows using the variational property and Minkowski identities(see \cite{GL-2015}). However, in the general warped product spaces instead of space forms, there are many difficulties to overcome. The main difficulty is that the long time existence of the flow or applications to geometric inequalities always require the preserving of some kind of convexity. In this paper, we will show that the static convexity is preserved along the flow \eqref{flow-WVPF} under an initial gradient assumption(see Theorem \ref{preservingofconvex} for details).

In the remaining part of this section, we deal with the weighted geometric inequalities for static convex domains in static rotationally symmetric spaces, hence we first give the following definitions:
\begin{defn}
	Let $\mathfrak{M}$ be a Riemannian manifold, we denote $g_{\mathfrak{M}},\nabla_{\mathfrak{M}},\nabla^2_{\mathfrak{M}},\Delta_{\mathfrak{M}}$ and $\text{Ric}_{\mathfrak{M}}$ as the metric, Levi-Civita connenction, Hessian operator, Laplacian operator and Ricci tensor on $\mathfrak{M}$ respectively. Additionally, we assume that $V$ is a smooth nontrival function on $\mathfrak{M}$.
	
	We say that a Riemannian triple $(\mathfrak{M},g_{\mathfrak{M}},V)$ is static if
	\begin{equation}\label{eq-static}
		\Delta_{\mathfrak{M}}{V}g_{\mathfrak{M}}-\nabla^2_{\mathfrak{M}}{V}+V\text{Ric}_{\mathfrak{M}}=0.
	\end{equation}
	Furthermore, We say that a Riemannian triple $(\mathfrak{M},g_{\mathfrak{M}},V)$ is sub-static, if there exists a smooth $(0,2)-$tensor $Q$ on $\mathfrak{M}$ such that
	\begin{equation}
		V\cdot Q:=\Delta_{\mathfrak{M}}{V}g_{\mathfrak{M}}-\nabla^2_{\mathfrak{M}}{V}+V\text{Ric}_{\mathfrak{M}}\geq 0.
	\end{equation}
	In both cases, we call $V$ a potential function.
\end{defn}
\begin{rem}
	By Lemma \ref{Lem-stacon} in \S \ref{sub-sta}, if $\mathbf{N}^{n+1}$ is a rotationally symmetric space equipped with the metric \eqref{eq-metric}, then it is equivalent that the Riemannian triple $(\mathbf{N}^{n+1},\bar{g},\phi')$ is static and $\mathbf{N}^{n+1}$ has constant scalar curvature. 
\end{rem}

Next we introduce a convexity condition which is named as static convex by Brendle-Wang\cite{BW-2014} for its correspondence in the static spacetime.

\begin{defn}
    Let $(\mathfrak{M},g_{\mathfrak{M}},V)$ be a sub-static Riemannian triple, a hypersurface $M$ of $\mathfrak{M}$ is called static convex(resp. strictly static convex), if the second fundamental form $\{h_{ij}\}$ satisfies
    \begin{equation}\label{In-sscon}
    h_{ij}\geq(\text{resp.}>)\frac{V_{,\nu}}{V} g_{ij}
    \end{equation}
    everywhere on $M$.
\end{defn}

In our case, we consider the static rotationally symmetric space $\mathbf{N}^{n+1}$ equipped with the metric \eqref{eq-metric}, then the potential function $V=\phi'$ satisfies the static condition \eqref{eq-static} and hence the static(resp. strictly static) convex inequality \eqref{In-sscon} is equivalent to 
\begin{equation}
    h_{ij}\geq(\text{resp.}>)\frac{u\phi''}{\phi\phi'} g_{ij}.
\end{equation}
Let $M$ be a closed, emdedded hypersurface in $\mathbf{N}^{n+1}$, then there are two cases: (i) $M$ is the boundary of a bounded domain $\Omega$; (ii) There is a bounded domain $\Omega$ with $\partial\Omega= S(r_0)\cup M$. In both cases, we say that $\Omega$ is a static convex(resp. strictly static convex) domain, if $M$ is static convex(resp. strictly static convex). We will show the existence of a strictly static convex point on a closed, embedded hypersurface $M$ in $\mathbf{N}^{n+1}$ in Proposition \ref{sconvpoint} unless $M=S(r_0)$ with $((\phi')^2-\phi\phi'')(r_0)=0$.

It's well-known that the static condition (\ref{eq-static}) guarantees that the Lorentzian warped product $-V^2dt\otimes dt+g_{\mathfrak{M}}$ is a solution of Einstein's equation. Moreover, static manifolds have been studied in connection with questions in general relativity (see e.g. \cites{CJ-2000,WWZ-2017}). Recently, some progress has been made in the geometric inequalities for domains in static or sub-static Riemannian triples. Brendle et.al.\cites{Brendle13,BHW16} established some of the inequalities in sub-static Riemannian triples with certain warped product structure. Li and Xia \cite{XL-2019} proved the Heintze-Karcher type inequality and Minkowski type inequality in the sub-static Riemannian triple  using their useful integral formula (built in \cite{QX-2015} firstly). Their graceful proofs require no assumptions on rotational symmetry or topological constraints for the ambient manifolds. 

The flow \eqref{flow-WVPF} was firstly introduced by Hu and Li\cite{HL-2019} in the hyperbolic space $\mathbb{H}^{n+1}$. For any bounded domain $\Omega$ with smooth boundary $M=\partial\Omega$. They defined the weighted curvature integrals as follows:
\begin{align*}
    W_0^{\phi'}(\Omega)&=\int_{M} u d\mu=\int_{\Omega} {(n+1)\phi'} dvol,\quad W_{n+1}^{\phi'}(\Omega)=\int_{M}{\phi' E_n}d\mu,\\
    W_k^{\phi'}(\Omega)&=\int_{M}{\phi' E_{k-1}}d\mu=\int_{M}{u E_k} d\mu,\quad k=1,\dots,n.
\end{align*}

Here $E_k:=\sigma_k/C_n^k$, is the kth normalized mean curvature of $M$. They proved the following convergence result:
\begin{thmA}[\cite{HL-2019}]
    Let $X_0$ be a smooth embedding of a closed $n$-dimensional manifold $M$ in $\mathbb{H}^{n+1}$ such that $M_0=X_0(M)$ is star-shaped. Then any solution $M_t=X(M,t)$ of \eqref{flow-WVPF} remains star-shaped for $t>0$ and it converges to a geodesic sphere $\partial B_{r_{\infty}}$ centered at the origin in the $C^{\infty}$-topology as $t\to\infty$, where the radius $r_{\infty}$ is uniquely determined by $W_0^{\phi'}(B_{r_{\infty}})=W_0^{\phi'}(\Omega_0)$. Moreover, if the initial hypersurface $M_0=X_0(M)$ is static convex, then the flow hypersurface $M_t=X(t,M)$ becomes strictly static convex for $t>0$.
\end{thmA}

The flow \eqref{flow-WVPF} in $\mathbb{H}^{n+1}$ has an advantage that it preserves $W_0^{\phi'}(\Omega_t)$. Meanwhile, they found that $W_k^{\phi'}(\Omega_t),k=1,\dots,n-1$ are monotone decreasing in time $t$ as long as the flow admits a static convex solution. Hence they proved the following inequalities between the weighted curvature integrals:
\begin{thmB}
    Let $\Omega$ be a static convex domain with smooth boundary $M$ in $\mathbb{H}^{n+1}$. For $1\leq k\leq n+1$, there holds
    \begin{equation}\label{in-wci}
        W_{k}^{\phi'}(\Omega)\geq h_k\circ h_0^{-1}(W_0^{\phi'}(\Omega))
    \end{equation}
    Equality holds in \eqref{in-wci} if and only if $\Omega$ is a geodesic ball centered at the origin. Here $h_k:[0,\infty)\to\mathbb{R}^{+}$ is a monotone function defined by $h_k(r)=W_k^{\phi'}(B_r)=\omega_n\sinh^{n+1-k}{r}\cosh^k{r}$, the $k$th weighted curvature integral for a geodesic ball of radius $r$, and $h_0^{-1}$ is the inverse function of $h_0$.
\end{thmB}

In this paper, we want to generalize their results to a broader class of ambient warped product spaces under mild assumptions on $\phi$. However, we cannot prove that the static convexity is preserved along the flow \eqref{flow-WVPF} in general warped product space with only static convexity required on $M_0$. In order to state our results, we introduce the following conception:
 
 \begin{defn}
     We say that a graphical hypersurface $M=\{(r(\theta),\theta)\in[r_0,\bar{r})\times\mathbb{S}^n\}$ is $\varepsilon$-close to a slice of $\mathbf{N}^{n+1}$ in the $C^1$ sense, if there exists a constant $0\leq\varepsilon<\infty$, such that
     \begin{align}\label{epcloser}
	|D r|^2 \leq \varepsilon \phi^2
\end{align}
    holds everywhere on $M$, where $D$ is the gradient with respect to the round metric $\sigma$ on $\mathbb{S}^n$. Furthermore, if we use a graphical representation for the support function $u$ and introduce a new variable $\gamma$( see \S \ref{subsection-graph} for a concrete definition), then condition \eqref{epcloser} is equivalent to
    \begin{equation}\label{epclose}
    |D\gamma|^2\leq \varepsilon \,\ \text{or} \,\ \frac{u^2}{\phi^2}\geq \frac{1}{1+\varepsilon}
    \end{equation}
 \end{defn}

Then we can prove that static convexity is preserved along the flow \eqref{flow-WVPF}, provided that the initial hypersurface $M_0$ is $\varepsilon_0$-close to a slice in $\mathbf{N}^{n+1}$ in the $C^1$ sense for some $\epsilon_0>0$.

\begin{thm}\label{preservingofconvex}
    Let $\mathbf{N}^{n+1}$ be a static rotationally symmetric space under the assumptions in Definition \ref{defn-N}. Assume that $M_0$ is a smooth, static convex and graphical hypersurface lies in $B(R)\subset \mathbf{N}^{n+1}$ for some $R>0$. Then there exists a constant $\epsilon_0>0$ depending only on $n$ and $R$, such that  if $M_0$ is $\epsilon_0$-close to a slice of $\mathbf{N}^{n+1}$ in the $C^1$ sence, then the flow \eqref{flow-WVPF} starting from $M_0$ remains to be static convex for all time $t\in[0,\infty)$. Moreover, the evolving hypersurface $M_t$ becomes strictly static convex for $t>0$ unless $M_t\equiv S(r_0)$ with $((\phi')^2-\phi\phi'')(r_0)=0$.
\end{thm}


\begin{rem}
The lower bound condition $(\phi')^2-\phi\phi''\geq 0$ is needed for the preserving of static convexity along the flow \eqref{flow-WVPF}. Meanwhile, since we focus on the static convex hypersurface in Theorem \ref{preservingofconvex}, we need to ensure that the convergent limit of the flow, i.e. a slice of $\mathbf{N}^{n+1}$ has its principal curvature $\kappa_i\geq\frac{u\phi''}{\phi\phi'}=\frac{\phi''}{\phi'}$, which is equivalent to the lower bound condition by Remark \ref{rem-slice}.
\end{rem}



Under the assumptions of Theorem \ref{preservingofconvex}, we know that along the flow \eqref{flow-WVPF}, $V_{\phi}^{\alpha}(\Omega_t)$($\alpha\leq 1$) are monotone increasing and $A_{i,\phi}(M_t)$($i=1,2$) are monotone decreasing according to Lemma \ref{lem-mono} and Proposition \ref{Prop-monto}. Hence as applications of Theorem \ref{preservingofconvex}, we obtain a family of weighted geometric inequalities for static convex domains.
\begin{cor}\label{Cor-wii}
Let $\mathbf{N}^{n+1}$ be a static rotationally symmetric space under the assumptions in Definition \ref{defn-N}. Given a constant $R>r_0$, then there exists a constant $\varepsilon>0$, depending only on $n$ and $R$, such that for any static convex domain $\Omega\subset B(R)$ which is bounded by $S(r_0)$ and a smooth graphical hypersurface $M$ with $M$ being $\varepsilon$-close to a slice of $\mathbf{N}^{n+1}$ in the $C^1$ sense, the following weighted geometric inequalities hold:
	\begin{align}
		A_{0,\phi}(M)&\geq \chi_{0,\alpha}(V_{\phi}^{\alpha}(\Omega)),\quad \text{if}\quad n\geq 2,\label{In-wii}\\
            A_{1,\phi}(M)&\geq \chi_{1,\alpha}(V_{\phi}^{\alpha}(\Omega)),\quad \text{if}\quad n\geq 3,\label{In-wiii}
	\end{align} 
	where $\chi_{i,\alpha} (i=0,1$\,\,and\,\, $\alpha\leq 1)$ are the functions defined in Equation \eqref{eq-chi}. The equalities hold in \eqref{In-wii} or \eqref{In-wiii} if and only if $M$ a slice of $\mathbf{N}^{n+1}$.
\end{cor} 

\begin{rem}
In the proof of deducing the monotonicities of $A_{i,\phi}(M_t)$ ($i=0,1$), we will use the Minkowski identities in Lemma \ref{lem-Min}, hence we need extra lower bounds on dimension $n$, and in this procedure, the upper bound condition $(\phi')^2-\phi\phi''\leq 1$ and the derivative condition $\phi''>0$ are also needed.
\end{rem}
Besides the standard spaces of constant sectional curvature, the most basic example of static, rotationally symmetric manifold is the Schwarzchild manifold. Let us fix a constant $m>0$. The Schwarzschild manifold is the $(n+1)$-dimensional ($n\geq 2$) manifold $\mathbf{N}^{n+1}=[0,+\infty)\times\mathbb{S}^n$ equipped with the metric
\begin{align}\label{schwazmetric}
    \bar{g}=dr^2+\lambda^2(r)\sigma,
\end{align}
where $\lambda(r):[0,+\infty)\to[s_0,+\infty)$ is the smooth function satisfies the ODE
\begin{align*}
    \lambda'(r)=\sqrt{1-2m\lambda^{1-n}},
\end{align*}
and $s_0$ is the unique positive solution of $1-2ms_0^{1-n}=0$. Then, 
\begin{align*}
    \lambda''(r)=(n-1)m\lambda^{-n},
\end{align*}
and
\begin{align*}
    (\lambda')^2(r)-\lambda(r)\lambda''(r)=1-m(n+1)\lambda^{1-n}.
\end{align*}
Consequently, if we denote $r_0$ the positive real number such that $1-m(n+1)\lambda^{1-n}(r_0)=0$ and restrict $M_0$ in the region $[r_0,\infty)\times\mathbb{S}^n$, all the conditions in Definition \ref{defn-N} can be satisfied in this case. Hence, the weighted geometric inequalities can be established in the Schwarzschild manifold as a direct consequence of Corollary \ref{Cor-wii}.
\begin{cor}\label{Cor-Sci}
    Assume that $\mathbf{N}^{n+1}(n\geq 2)$ is the Schwarzschild manifold, $B(R)\subset \mathbf{N}^{n+1}(R>r_0)$ is the domain bounded by $S(r_0)$ and $S(R)$. Then there exists a constant $\varepsilon>0$, depending only on $n$ and $R$, such that for any static convex domain $\Omega\subset B(R)$ which is bounded by $S(r_0)$ and a smooth graphical hypersurface $M$ with $M$ being $\varepsilon$-close to a slice of $\mathbf{N}^{n+1}$ in the $C^1$ sense, the following weighted geometric inequalities holds:
	\begin{equation}
		A_{0,\phi}(M)\geq \chi_{0,\alpha}(V_{\phi}^{\alpha}(\Omega)).\label{In-scwi}
	\end{equation} 
 If we assume further that $n\geq 3$, then 
 \begin{equation}
     A_{1,\phi}(M)\geq \chi_{1,\alpha}(V_{\phi}^{\alpha}(\Omega)),\label{In-scwii}
 \end{equation}
where $\chi_{i,\alpha} (i=0,1$\,\,and\,\, $\alpha\leq 1)$ are the functions defined in Equation \eqref{eq-chi}. The equalities hold in \eqref{In-scwi} or \eqref{In-scwii} if and only if $M$ a slice of $\mathbf{N}^{n+1}$.
\end{cor}
\begin{rem}
    The result also holds in the anti-de Sitter- Schwarzchild manifold by the totally similar calculation.
\end{rem}
The paper is organized as follows: 
In \S\ref{sec-pre}, we collect some preliminaries including the geometry of the static rotationally symmetric space $\mathbf{N}^{n+1}$, the graphical representation for hypersurfaces in $\mathbf{N}^{n+1}$, and the evolution equations along the flow \eqref{flow-WVPF}. In \S\ref{sec-long}, we rewrite the flow equation \eqref{flow-WVPF} as a scalar parabolic initial value problem \eqref{flow-QF} and give a priori $C^0$ and $C^1$ estimates, then we complete the proof of Theorem \ref{Thm-main} and Corollary \ref{Cor-VVphi}. In \S\ref{Sec-static-preserving}, we use the tensor maximum principle to show that the static convexity is preserved along the flow \eqref{flow-WVPF}, provided that the initial hypersurface $M_0$ is $\varepsilon_0$-close to a slice of $\mathbf{N}^{n+1}$ in the $C^1$ sense for some constant $\varepsilon_0>0$ and then complete the proof of Theorem \ref{preservingofconvex}. Finally, in \S\ref{sec-inequality}, we show that $A_{i,\phi}(M_t)(i=0,1)$ are monotone decreasing along the flow \eqref{flow-WVPF} under the assumptions of Theorem \ref{preservingofconvex} and complete the proof of Corollary \ref{Cor-wii}.

\begin{ack}
	The research was surpported by National Key R and D Program of China 2021YFA1001800 and 2020YFA0713100, NSFC 11721101, China Postdoctoral Science Foundation No.2022M723057, the Fundamental Research Funds for the Central Universities and Shuimu Tsinghua Scholar Program (No. 2023SM102).
\end{ack}  
\section{Preliminaries}\label{sec-pre}
\subsection{The geometric quantities of the rotationally symmetric space}\label{sub-gq}
In this subsection, we give the formulas of curvatures in a rotationally symmetric space $\mathbf{N}^{n+1}$ equipped with the metric \eqref{eq-metric} and show some basic properties of the conformal vector field $V=\phi(r)\partial_r$ defined in $\mathbf{N}^{n+1}$. These results are well-known and most of them can be found in \cites{Besse, Brendle13, Ding-2011, GLW-2019}.

Let $\mathbf{N}^{n+1}$ be a rotationally symmetric space equipped with the metric \eqref{eq-metric} and $\theta=\{\theta^1,\cdots,\theta^n\}$ be the local coordinate of $\mathbb{S}^n$. Then $\{\partial_i:={\partial_{\theta^i}}\}_{i=1}^{n}$ form a basis of the tangent space of $\mathbb{S}^n$ and $\{\partial_r:=\frac{\partial}{\partial r}\} \cup \{\partial_i:={\partial_{\theta^i}}\}_{i=1}^{n}$ form a basis of the tangent space of $N$.

Assume that M is a hypersurface of $N$. Let $D,\nabla,\overline{\nabla}$ and $\Delta_0,\Delta_g,\overline{\Delta}$ be the Levi-Civita connection and Laplacian operator of $\mathbb{S}^n, M$ and $N$, respectively. We denote $R$ and $\bar{R}$, $Ric$ and  $\overline{Ric}$, $\rho$ and $\bar{\rho}$ as the Riemannian curvature tensor, Ricci tensor and the scalar curvature of $M$ and $\mathbf{N}^{n+1}$ respectively. Our convention for the Riemannian curvature tensor is:
\begin{equation*}
	\bar{R}(X,Y)Z=\overline{\nabla}_X\overline{\nabla}_Y Z-\overline{\nabla}_Y\overline{\nabla}_X Z-\overline{\nabla}_{[X,Y]} Z
\end{equation*}
and the purely covariant version is
\begin{equation*}
	\bar{R}(X,Y,Z,W)=\bar{g}(\bar{R}(Z,W)Y,X).
\end{equation*}
\begin{prop}\label{prop-RT}
	The Riemannian curvature tensors of $(\mathbf{N}^{n+1},\bar{g})$ have the following components:
	\begin{align}
         \bar{R}_{ijkl}:=\bar{R}(\partial_i,\partial_j,\partial_k,\partial_{\ell})&=-\frac{(\phi')^2-1}{\phi^2}(\bar{g}_{ik}\bar{g}_{jl}-\bar{g}_{il}\bar{g}_{jk}),\label{eq-Rijkl}\\
         \bar{R}_{irjr}:=\bar{R}(\partial_i,\partial_r,\partial_j,\partial_r)&=-\frac{\phi''}{\phi}\bar{g}_{ij}\label{eq-Rirjr},
	\end{align}
where $i,j,k,\ell=1,2,\cdots,n$ and other components of the Riemannian curvature tensors are equal to 0. 

Furthermore, if $X,Y,Z,W\in T\mathbf{N}^{n+1}$, then we have
\begin{align}
	\bar{R}(X,Y,Z,W)=&\frac{(\phi')^2-\phi\phi''-1}{\phi^2}(a_X a_Z \bar{g}(Y,W)+a_Y a_W \bar{g}(X,Z)-a_X a_W \bar{g}(Y,Z)-a_Y a_Z\bar{g}(X,W))\notag\\
	&-\frac{(\phi')^2-1}{\phi^2}(\bar{g}(X,Z)\bar{g}(Y,W)-\bar{g}(X,W)\bar{g}(Y,Z)),\label{eq-gnR}
\end{align} 
where $a_X=\bar{g}(X,\partial_r)$ and $a_Y, a_Z, a_W$ are defined as the same form of $a_X$.
\end{prop}
\begin{proof}
Equations \eqref{eq-Rijkl} and \eqref{eq-Rirjr} can be found in \cite{Ding-2011}. Hence we just calculate $\bar{R}(X,Y,Z,W)$ for arbitrary $X,Y,Z,W\in T\mathbf{N}^{n+1}$. 

First, we decompose $X$ as follows: 
\begin{equation*}
X=a_X \partial_r+X^{\perp},
\end{equation*}
where $a_X=\bar{g}(X,\partial_r)$ and $X^{\perp}$ denotes the component of $X$ which is orthogonal to the radial direction. The decomposition of $Y,Z$ and $W$ can be done similarly. Combining the decompositions with equations \eqref{eq-Rijkl} and \eqref{eq-Rirjr}, we calculate as follows:
\begin{align*}
	\bar{R}(X,Y,Z,W)
=&\bar{R}(a_X \partial_r+X^{\perp},a_Y \partial_r+Y^{\perp},a_Z \partial_r+Z^{\perp},a_W \partial_r+W^{\perp})\notag\\
=&a_X a_Z\bar{R} (\partial_r,Y^{\perp},\partial_r,W^{\perp})+a_X a_W\bar{R}(\partial_r,Y^{\perp},Z^{\perp},\partial_r)\notag\\
&+a_Y a_Z\bar{R} (X^{\perp},\partial_r,\partial_r,W^{\perp})+a_Y a_W\bar{R} (X^{\perp},\partial_r,Z^{\perp},\partial_r)\notag\\
&+\bar{R}(X^{\perp},Y^{\perp},Z^{\perp},W^{\perp})\notag\\
=&-\frac{\phi''}{\phi}a_X a_Z\bar{g}(Y^{\perp},W^{\perp})+\frac{\phi''}{\phi}a_X a_W\bar{g}(Y^{\perp},Z^{\perp})\notag\\
&+\frac{\phi''}{\phi}a_Y a_Z\bar{g}(X^{\perp},W^{\perp})-\frac{\phi''}{\phi}a_Y a_W\bar{g}(X^{\perp},Z^{\perp})\notag\\
&-\frac{(\phi')^2-1}{\phi^2}(\bar{g}(X^{\perp},Z^{\perp})\bar{g}(Y^{\perp},W^{\perp})-\bar{g}(Y^{\perp},Z^{\perp})\bar{g}(X^{\perp},W^{\perp})).
\end{align*}
Since
\begin{align*}
	\bar{g}(X^{\perp},Z^{\perp})&=\bar{g}(X-a_X\partial_r,Z-a_Z\partial_r)=\bar{g}(X,Z)-a_X a_Z
\end{align*}
and other terms can be calculated similarly, then equation \eqref{eq-gnR} follows easily.
\end{proof}

The following formula of the Ricci tensor $\overline{Ric}$ can be found in \cite{GLW-2019} (see also \cites{Besse,Brendle13}) and hence the formula of the scalar curvature $\bar{\rho}$ follows by a direct calculation.

\begin{lem}
   Let $\mathbf{N}^{n+1}$ be a rotationally symmetric space equipped with the metric \eqref{eq-metric}. Then the Ricci tensor $\overline{Ric}$ and the scalar curvature $\bar{\rho}$ of $(\mathbf{N}^{n+1},\bar{g})$ are given by
\begin{align}
	\overline{Ric}&=-n\frac{\phi''}{\phi} dr^2-[(n-1)((\phi')^2-1)+\phi\phi'']\sigma,\label{eq-RicciT}\\
	\bar{\rho}&=-n\left(2\frac{\phi''}{\phi}+(n-1)\frac{(\phi')^2-1}{\phi^2}\right).\label{eq-ScalarC}
\end{align}
\end{lem}

Let $V:=\phi(r)\partial_r$ be the vector field defined on $\mathbf{N}^{n+1}$, then it is well-known that $V$ is a conformal Killing vector field and have the following basic properties.

\begin{lem}[\cites{Brendle13,GLW-2019}]\label{lem-Phi}
    The vector field $V=\phi(r)\partial_r$ is the conformal vector field on $\mathbf{N}^{n+1}$, i.e., the Lie derivative of $V$ is given by $\mathcal{L}_V\bar{g}=2\phi'(r)\bar{g}$. Moreover,
    \begin{equation}\label{eq-conformal}
    \langle\bar{\nabla}_X{V},Y\rangle=\phi'\bar{g}(X,Y)
    \end{equation}
 for any $X,Y\in T\mathbf{N}^{n+1}$. Furthermore, if we let $\Phi'(r)=\phi(r)$, then on a hypersurface $M\subset\mathbf{N}^{n+1}$,
    \begin{align}
\Phi_{ij}&=\nabla_j\langle{\phi\partial_r,e_i}\rangle=\phi'g_{ij}-uh_{ij},\label{eq-Phiij}\\
        \Delta_g\Phi&=n\phi'-uH.\label{eq-DeltaPhi}
    \end{align}
    where $\Phi_{ij}$ is the Hessian of the function $\Phi$ and $\Delta_g\Phi$ is the Laplacian of the function $\Phi$.
\end{lem}

\subsection{Graphical representation for hypersurfaces in $\mathbf{N}^{n+1}$}\label{subsection-graph}
Let $\mathbf{N}^{n+1}$ be a rotationally symmetric space equipped with the metric \eqref{eq-metric} and $M$ be a hypersurface of $\mathbf{N}^{n+1}$. If $M$ is a graphical hypersurface of $\mathbf{N}^{n+1}$, i.e., $M$ is a graph of a smooth and positive function $r(\theta)$ on $\mathbb{S}^n$:
\begin{equation*}
	M=\{(r(\theta),\theta)\in[r_0,\bar{r})\times\mathbb{S}^n\}
\end{equation*}
As mentioned in \S \ref{sub-gq}, we denote $\sigma$ as the round metric on $\mathbb{S}^n$, $D$ as the Levi-Civita connection on $\mathbb{S}^n$ with respect to $\sigma$ and $g$ as the induced metric on $M$. Meanwhile, for a smooth function $f$ defined on $M$, we set $f_i=D_i f, f^j=\sigma^{ij}f_i$ and $f_{ij}=D_iD_j f$, where $(\sigma^{ij})$ is the inverse matrix of $(\sigma_{ij})$. 
Then the tangent space of $M$ is spanned by $\{e_i=\partial_i+r_i\partial_r,i=1,2,\cdots,n\}$. As Guan and Li did in \cite{GL-2015}, 
we introduce a new variable $\gamma$, which satisfies
\begin{equation}\label{de-gamma}
	\frac{d\gamma}{d r}=\frac{1}{\phi(r)}
\end{equation}
and let $\omega=\sqrt{1+|D\gamma|^2}$. Hence, we have the following radial graph representation for  some geometric quantities of the hypersurfaces $M$ (see \cite{GL-2015} for details):
\begin{align}
	g_{ij}=&\phi^2(\sigma_{ij}+\gamma_i\gamma_j),\ 
	g^{ij}=\frac{1}{\phi^2}(\sigma^{ij}-\frac{\gamma^i\gamma^j}{\omega^2}),\label{eq-g1}\\
	\nu=&\frac{1}{\omega}(\partial_r-\frac{\gamma^j\partial_j}{\phi}),\
	u=\frac{\phi}{\omega},\label{eq-nu1}\\
	h_{ij}=&\frac{\phi}{\omega}(-\gamma_{ij}+\phi'\gamma_i\gamma_j+\phi'\sigma_{ij}),\label{eq-hij1}\\
	h_j^i=&\frac{1}{\phi\omega}(\sigma^{ik}-\frac{\gamma^i\gamma^k}{\omega^2})(-\gamma_{kj}+\phi'\gamma_k\gamma_j+\phi'\sigma_{kj}).\label{eq-WG1}
\end{align}
\begin{rem}\label{rem-slice}
    A slice of $\mathbf{N}^{n+1}$ has its second fundamental form $h_{ij}=\frac{\phi'}{\phi}g_{ij}$ and hence the principal curvatures are $\kappa_i=\frac{\phi'}{\phi}$ according to equations \eqref{eq-g1},\eqref{eq-hij1} and \eqref{eq-WG1}.
\end{rem}
We now consider the flow equation \eqref{flow-WVPF} of radial graphs over $\mathbb{S}^n$ in $\mathbf{N}^{n+1}$. It is known (cf. \cite{Ger06}) that if a closed hypersurface which is graphical and satisfies 
\begin{equation*}
	\partial_t X=\mathcal{F}\nu,
\end{equation*}
then the evolution of the scalar function $r=r(X(z,t),t)$ satisfies
\begin{equation*}
	\partial_t r=\mathcal{F}\omega
\end{equation*}
and hence $\gamma$ satisfies 
\begin{equation*}
	\partial_t\gamma=\mathcal{F}\frac{\omega}{\phi}
\end{equation*}
Therefore, the flow \eqref{flow-WVPF} is equivalent to the following parabolic initial value problem:
\begin{equation}\label{flow-SWVPF}
	\left\{\begin{aligned}
		\partial_t\gamma=&~(n-\frac{uH}{\phi'})\frac{\omega}{\phi}\quad \text{for}\quad  (z,t)\in\mathbb{S}^n\times[0,\infty),\\
		\gamma(\cdot,0)=&~\gamma_0(\cdot).
	\end{aligned}\right.
\end{equation}

\subsection{Static Riemannian manifolds}\label{sub-sta}
In this subsection, we prove some basic properties of static rotationally symmetric spaces, which will be of great use in proving the preserving of static convexity along the flow \eqref{flow-WVPF}.  

It is well-known that a static Riemannian triple must be of constant scalar curvature, see e.g \cite{CJ-2000}. In fact, if $\mathbf{N}^{n+1}$ is a rotationally symmetric space equipped with the metric \eqref{eq-metric}, then it is equivalent that $\mathbf{N}^{n+1}$ is static and it has constant scalar curvature. We present a simple proof here.
\begin{lem}\label{Lem-stacon}
	Let $\mathbf{N}^{n+1}$ be a rotationally symmetric space equipped with the metric \eqref{eq-metric}, then the Riemannian triple $(\mathbf{N}^{n+1},\bar{g},\phi')$ being sub-static is equivalent to that the scalar curvature $\bar{\rho}$
	is non-increasing in $r$ or precisely,
	\begin{equation}\label{eq-subst}
		\phi^2\phi'''+(n-2)\phi\phi'\phi''-(n-1)\phi'((\phi')^2-1)\geq 0.
	\end{equation}
	In particular, the Riemannian triple $(\mathbf{N}^{n+1},\bar{g},\phi')$ being static is equivalent to that the scalar curvature $\bar{\rho}$ is a constant not depending on $r$ or precisely,
	\begin{equation}\label{eq-sTa}
		\phi^2\phi'''+(n-2)\phi\phi'\phi''-(n-1)\phi'((\phi')^2-1)=0.
	\end{equation}
Moreover, there exists a constant $C_0$ such that 
\begin{align}
    (\phi')^2-\phi\phi''-1&\leq-\frac{C_0}{\phi^{n-1}},\quad \text{if\, $\mathbf{N}^{n+1}$ is sub-static,}\label{eq-subst-2}\\
    (\phi')^2-\phi\phi''-1&=-\frac{C_0}{\phi^{n-1}},\quad \text{if\, $\mathbf{N}^{n+1}$ is static.}\label{eq-stat-2}
\end{align}
In fact, the constant $C_0$ can be chosen as $-\phi^{n-1}(r_0)((\phi')^2-\phi\phi''-1)(r_0)$ in the sub-static case and $C_0$ is an universal constant not depending on the choice of $r_0$ in the static case. Furthermore, in the static case, $C_0=0$ if and only if $\mathbf{N}^{n+1}$ is of constant sectional curvature.
\end{lem}
\begin{rem}
    As is shown in \cite{Kobayashi-1982}, the geometric meaning of the constant $C_0$ for the static case is as follows: By $\eqref{eq-RicciT}$, the Ricci tensor has, at each point of $\mathbf{N}^{n+1}$, two eigenvalues $\lambda_1=-n\frac{\phi''}{\phi}$ and $\lambda_2=-\frac{(n-1)((\phi')^2-1)+\phi\phi''}{\phi^2}$  with multiplicities $1$ and $n$ respectively. Then by $\eqref{eq-stat-2}$, we have 
    \begin{equation}
        \lambda_1-\lambda_2=(n-1)\frac{(\phi')^2-\phi\phi''-1}{\phi^2}=-(n-1)\frac{C_0}{\phi^{n+1}}.
    \end{equation}
\end{rem}
\begin{proof}
By a direct calculation, see e.g. \cite{Brendle13},
	we have
	\begin{align*}
		&(\overline{\Delta}{\phi'})\bar{g}-\overline{\nabla}^2{\phi'}+\phi'\overline{Ric}\notag\\
		=&\left[\phi^2\phi'''+(n-2)\phi\phi'\phi''-(n-1)\phi'((\phi')^2-1)\right]\sigma\\
		=&\left[\frac{1}{2}\phi^3\frac{d}{dr}\left(2\frac{\phi''}{\phi}+(n-1)\frac{(\phi')^2-1}{\phi^2}\right)\right]\sigma.
	\end{align*}
Comparing the above equation with the expression formula (\ref{eq-ScalarC}) of the scalar curvature $\bar{\rho}$, we now have
\begin{align*}
  &(\overline{\Delta}{\phi'})\bar{g}-\overline{\nabla}^2{\phi'}+\phi'\overline{Ric}= -\frac{\phi^3}{2n}\left(\frac{d}{d r}\bar{\rho}\right)\sigma,
\end{align*}
	which yields (\ref{eq-subst}) and (\ref{eq-sTa}).
Next, note that 
\begin{align}
    \frac{d}{dr}\left(\phi^{n-1}((\phi')^2-\phi\phi''-1)\right)=-\phi^{n-2}\left[\phi^2\phi'''+(n-2)\phi\phi'\phi''-(n-1)\phi'((\phi')^2-1)\right].\label{eq-ddr}
\end{align}
If $\mathbf{N}^{n+1}$ is sub-static, then (\ref{eq-subst}) implies that there exists a constant $C_0=-\phi^{n-1}(r_0)((\phi')^2-\phi\phi''-1)(r_0)$, such that
\begin{align}\label{In-ssinq}
    (\phi')^2-\phi\phi''-1\leq-\frac{C_0}{\phi^{n-1}},\,\forall r\in [r_0,\bar{r}).
\end{align}
If $\mathbf{N}^{n+1}$ is static, combining \eqref{eq-ddr} with \eqref{eq-sTa}, we have
\begin{equation*}
     \frac{d}{dr}\left(\phi^{n-1}((\phi')^2-\phi\phi''-1)\right)=0
\end{equation*}
and hence
\begin{equation*}
    \phi^{n-1}((\phi')^2-\phi\phi''-1)\equiv \text{const}:=-C_0
\end{equation*}
where $C_0$ is an universal constant not depending on the choice of $r_0$.

In the static case, if $C_0=0$ in \eqref{eq-stat-2}, i.e. $(\phi')^2-\phi\phi''-1\equiv0$. Then for $\forall X, Y, Z, W\in TN^{n+1}$, by equation \eqref{eq-gnR} we have
\begin{equation}\label{eq-RC}
	\bar{R}(X,Y,Z,W)=-\frac{\phi''}{\phi}(\bar{g}(X,Z)\bar{g}(Y,W)-\bar{g}(X,W)\bar{g}(Y,Z)).
\end{equation}
Moreover, we have $\phi'\phi''=\phi\phi'''$, which implies  
\begin{equation}\label{eq-phi''phi}
	\frac{d}{dr}\left(\frac{\phi''}{\phi}\right)=0
\end{equation}
Combing \eqref{eq-phi''phi} with \eqref{eq-RC} yields that $\mathbf{N}^{n+1}$ is of constant sectional curvature. 

On the other hand, if $\mathbf{N}^{n+1}$ is of constant sectional curvature $c$, we know that(cf. \cite{Petersen-2016}) $\mathbf{N}^{n+1}$ is either the Euclidean space $\mathbb{R}^{n+1}$ for $c=0$, the sphere $S^{n+1}\left(\frac{1}{\sqrt{c}}\right)$ for $c>0$ or the hyperbolic space $\mathbb{H}^{n+1}\left(\frac{1}{\sqrt{-c}}\right)$ for $c<0$. All of them are rotationally symmetric space equipped with the metric \eqref{eq-metric} with $\phi=r,\sin(\sqrt{c}r/)\sqrt{c}$ or $\sinh(\sqrt{-c}r)/\sqrt{-c}$ respectively. A direct calculation shows that $(\phi')^2-\phi\phi''-1\equiv 0$ in all of the three cases. This completes the proof of Lemma \ref{Lem-stacon}.
\end{proof}
Next, in the static rotationally symmetric space $\mathbf{N}^{n+1}$, we prove that there exists at least one strictly static convex point on any closed hypersurface which embedded in $\mathbf{N}^{n+1}$. This property will play a crucial role in the proof of the fact that the solution hypersurface $M_t$ becomes strictly static convex along the flow (\ref{flow-WVPF}) with static convex initial data $M_0$. 

In order to state the result clearly, we first give a specific definition of the unit outward normal vector field $\nu$. For any closed, embedded, orientable and connected hypersurface $M$ in $(\mathbf{N}^{n+1},\bar{g})$, $\mathbf{N}^{n+1}\setminus M$ has exactly two connected components and we denote the one which dose not contain $S(\bar{r})$ as $\Omega$ (if $\bar{r}=+\infty$, then we choose the bounded one). We either have $\partial\Omega=M$ or $\partial\Omega=M\cup S(r_0)$. In either case, let $\nu$ denote the unit normal vector field of $M$ pointing outward of $\Omega$.
\begin{prop}\label{sconvpoint}
    Let $\mathbf{N}^{n+1}$ be a static rotationally symmetric space satisfying $\phi, \phi'>0$ and $(\phi')^2-\phi\phi''\geq0$. Assume that $M$ is a closed, embedded hypersurface in $\mathbf{N}^{n+1}$, then there exists at least one strictly static convex point on $M$  unless $M=S(r_0)$ with $((\phi')^2-\phi\phi'')(r_0)=0$.
\end{prop}
\begin{proof}
     Consider the function $\phi^2$ defined on $M$, at the maximum point $p$ of $\phi^2$, we have
     \begin{align*}
         \nabla_i\phi^2=2\phi'\bar{g}(\phi\partial_r,e_i)=0,\ \forall i=1,\cdots,n,
     \end{align*}
then $\nu$ is parallel to the direction $\partial_r$ at the point $p$ due to $\phi,\phi'>0$. Since $r(p)=\max_{x\in M}r(x)$, the enclosed domain $\Omega$ is contained in $B(r(p))$. Thus, at the point $p$, we have $\nu=\partial_r$ by our convention for the choice of the unit outward normal, and hence $u=\phi$ at $p$. 

Moreover, using (\ref{eq-Phiij}) we have
\begin{align*}
   0\geq\nabla_j\nabla_i\phi^2=2\phi'(\phi'g_{ij}-uh_{ij}).
\end{align*}
Hence, at the point $p$, there holds
\begin{align}\label{strconv}
    h_{ij}\geq\frac{\phi'}{u}g_{ij}=\frac{\phi'}{\phi}g_{ij}.
\end{align}
This means that all the principal curvatures are not smaller than $\frac{\phi'}{\phi}$ at the maximum point $p$ of the function $\phi^2$. If $\phi''(p)<0$, it's obviously that the point $p$ is a strictly static convex point due to $\phi,\phi'>0$. Without loss of generality, we assume that $\phi''(p)\geq0$.
Since $\mathbf{N}^{n+1}$ is static, by (\ref{eq-stat-2}) we have
\begin{align*}
    (\phi')^2-\phi\phi''=1-\frac{C_0}{\phi^{n-1}},
\end{align*}
for some constant $C_0$.

If $C_0\leq0$, we have $(\phi')^2-\phi\phi''\geq1$, then
$\frac{\phi'}{\phi}>\frac{\phi''}{\phi'}$.
 Thus,  $h_{ij}>\frac{u\phi''}{\phi\phi'}g_{ij}$ at the point $p$ by (\ref{strconv}). 
 
 If $C_0>0$, we observe that the function $(\phi')^2-\phi\phi''$ is strictly increasing with respect to $r$. 
 If $M$ is not the slice $S(r_0)$, then
\begin{align*}
    ((\phi')^2-\phi\phi'')(p)>((\phi')^2-\phi\phi'')(r_0)\geq0.
\end{align*}
Therefore, by (\ref{strconv}) we can see that $p$ is a strictly static convex point. 
If $M=S(r_0)$, we have $h_{ij}=\frac{\phi'}{\phi}(r_0)g_{ij}$ on $M$ by Remark \ref{rem-slice}. Hence $h_{ij}>\frac{\phi''}{\phi'}g_{ij}$  holds everywhere on $M$, unless $((\phi')^2-\phi\phi'')(r_0)=0$. 
In conclusion, there exists at least one strictly static convex point $p$ on $M$  unless $M=S(r_0)$ with $((\phi')^2-\phi\phi'')(r_0)=0$.
\end{proof}
\begin{rem}
    In the case that $\mathbf{N}^{n+1}$ is the hyperbolic space $\mathbb{H}^{n+1}$, the inequality (\ref{strconv}) yields that $h_{ij}>g_{ij}$ at the point $p$. This proves the well-known fact that on any closed hypersurface in $\mathbb{H}^{n+1}$, there exists at least one strictly horo-convex point.
\end{rem}

\subsection{Minkowski identities and evolution equations} In this subsection, we will present some known facts including the Minkowski identities for hypersurfaces in $\mathbf{N}^{n+1}$ and evolution equations along the flow \eqref{flow-WVPF} which will be used later in this paper.

\begin{lem}\label{lem-Min}
Let $\mathbf{N}^{n+1}$ be be a rotationally symmetric space equipped with the metric \eqref{eq-metric} and $M$ be a closed hypersurface of $\mathbf{N}^{n+1}$. Then the following Minkowski identities hold:
    \begin{align}
        n\int_{M}{\phi'}d\mu&=\int_{M}{uH} d\mu,\label{eq-Min1}\\
        (n-1)\int_{M}{\phi' H}d\mu&=2\int_{M}{\sigma_2 u}d\mu+\int_{M}{u\left(\overline{Ric}(\partial_r,\partial_r)-\overline{Ric}(\nu,\nu)\right)}d\mu\label{eq-Min2}.
    \end{align}
In general, for any $k=1,2,\cdots,n$ we have
\begin{align}\label{eq-Min3}
     (n-k)\int_{M}{\phi' \sigma_{k}}d\mu&=(k+1)\int_{M}{ \sigma_{k+1}}d\mu+\frac{n-k}{n-1}\int_{M}\sigma_{k}^{ij}\overline{Ric}(e_i,\nu)\langle\phi\partial_r,e_j\rangle d\mu.
\end{align}
\end{lem}
\begin{proof}
    The identities (\ref{eq-Min1}) and (\ref{eq-Min2}) can be found in \cites{GLW-2019,LP-2023}. Next, we will give a simple proof of the general identity (\ref{eq-Min3}). 

   For any fixed  $k\in\{1,2,\cdots,n\}$, denote the vector field $W_k=\sigma_k^{ij}\langle\phi\partial_r,e_i\rangle e_j$, where $\sigma_k^{ij}=\frac{\partial\sigma_k}{\partial h_{ij}}$. Then we have
    \begin{align*}
    div_{M}W_{k+1}=&\sigma^{ij}_{k+1,j}\langle\phi\partial_r,e_i\rangle+\sigma_{k+1}^{ij}(\phi'g_{ij}-uh_{ij})\nonumber\\
        =&\sigma^{ij}_{k+1,j}\langle\phi\partial_r,e_i\rangle+(n-k)\phi'\sigma_{k}-(k+1)u\sigma_{k+1}.
    \end{align*}
Integrating the above equality, we immediately obtain
\begin{align}\label{Min3-eq1}
    (n-k)\int_{M}\phi'\sigma_{k}d\mu=(k+1)\int_{M}{ \sigma_{k+1}}d\mu-\int_{M}\sigma^{ij}_{k+1,j}\langle\phi\partial_r,e_i\rangle d\mu.
\end{align}
Recall that Brendle and Eichmair gave a concrete expression of $\sigma^{ij}_{k+1,j}\langle\phi\partial_r,e_i\rangle$  as follows (see \cite{BE13}*{Proposition 8}):
\begin{align}\label{Min3-eq2}
   \sigma^{ij}_{k+1,j}\langle\phi\partial_r,e_i\rangle=-\frac{n-k}{n-1}\sigma_{k}^{ij}\overline{Ric}(e_i,\nu)\langle\phi\partial_r,e_j\rangle.
\end{align}
Substituting (\ref{Min3-eq2}) into (\ref{Min3-eq1}), then the identity (\ref{eq-Min3}) follows.
\end{proof}
Next, we derive the evolution equations of some geometric quantities along a general flow
\begin{equation}\label{flow-general}
\partial_t X=\mathcal{F}\nu,
\end{equation}
where $\mathcal{F}$ is the speed function.

\begin{prop}\label{prop-evolution}
    Along the flow \eqref{flow-general},
 the induced metric $g_{ij}$, the unit outward normal $\nu$, the second fundamental form $h_{ij}$, the mean curvature $H$ and the area element $d\mu_t$ of the flow hypersurface $M_t$ evolve as follows: (see e.g. \cite{GLW-2019}):
\begin{align}
    \partial_tg_{ij}&=2\mathcal{F}h_{ij},\label{eveq-metric}\\
    \partial_t\nu&=-\nabla\mathcal{F},\label{eveq-normal}\\
    \partial_th_{ij}&=-\nabla_i\nabla_j\mathcal{F}+\mathcal{F}(h^2)_{ij}-\mathcal{F}\bar{R}_{i\nu j\nu},\label{eveq-h1}\\
    \partial_t H&=-\Delta_g{\mathcal{F}}-\mathcal{F}|A|^2-\mathcal{F}\overline{Ric}(\nu,\nu),\label{eveq-H}\\
    \partial_td\mu_t&=\mathcal{F}H d\mu_t\label{eq-area}
\end{align}
and hence if we denote $\Omega_t$ as the domain bounded by the hypersurface $M_t$ and the slice $S(r_0)$ in $\mathbf{N}^{n+1}$, then we have the following evolution equations for the weighted area, the weighted mean curvature integral of $M_t$ and the weighted $\alpha$-volume of $\Omega_t$:
\begin{align}
    \partial_t A_{0,\phi}(M_t) &=\int_{M_t}{(\frac{u\phi''}{\phi}+\phi'H)\mathcal{F}}\,d\mu_t,\label{va-Aphi}\\
    \partial_t A_{1,\phi}(M_t)&=\int_{M_t}\left({\frac{u\phi''}{\phi}H-\Delta_g{\phi'}+2\phi'\sigma_2-\phi'\overline{Ric}(\nu,\nu)}\right)\mathcal{F}\,d\mu_t, \label{va-Hphi}\\
    \partial_t V_{\phi}^{\alpha}(\Omega_t)&=\int_{M_t}{(\phi')^{\alpha}\mathcal{F}}\,d\mu_t.\label{va-Vphi}
\end{align}
\end{prop}

In this paper, we deal with the case that the speed function $\mathcal{F}=n-\frac{uH}{\phi'}$. Combining Lemma \ref{lem-Min} with the evolution equation \eqref{va-Vphi}, we have the following monotonicities for $V_{\phi}^{\alpha}(\Omega_t)$ along the flow \eqref{flow-WVPF}.

\begin{lem}\label{lem-mono}
     Let $\mathbf{N}^{n+1}$ be a rotationally symmetric space equipped with the metric \eqref{eq-metric} and $M_t$ is the smooth graphical solution of the flow \eqref{flow-WVPF}. Denote by $\Omega_t$ the domain enclosed by the hypersurface $M_t$ and the slice $S(r_0)$. Then along the flow \eqref{flow-WVPF}, 
     \begin{itemize}
         \item[(i)] $V_{\phi}(\Omega_t)$ keeps invariant.
         \item[(ii)] If $\phi''(\alpha-1)<0$, then $V_{\phi}^{\alpha}(\Omega_t)$ is monotone increasing and is strictly incresing unless $M_t$ is a slice of $\mathbf{N}^{n+1}$.
         \item[(iii)] If $\phi''(\alpha-1)>0$, then $V_{\phi}^{\alpha}(\Omega_t)$ is monotone decreasing and is strictly decresing unless $M_t$ is a slice of $\mathbf{N}^{n+1}$.
     \end{itemize}
\end{lem}

\begin{proof}
First, combining \eqref{eq-Min1} with \eqref{va-Vphi}, we have
\begin{equation}\label{eq-Vphi}
    \partial_t V_{\phi}(\Omega_t)=\int_{M_t}{(n\phi'-uH)} d\mu_t=0
\end{equation}
for $\alpha=1$ and hence $V_{\phi}(\Omega_t)$ keeps invariant along the flow.

For general $\alpha$, we calculate as follows:
\begin{align}
    \partial_t V_{\phi}^{\alpha}(\Omega_t)&=\int_{M_t}{(\phi')^{\alpha-1}\Delta_g\Phi}\,d\mu_t\notag\\
    &=-\int_{M_t}{\langle\nabla{(\phi')^{\alpha-1}},\nabla\Phi\rangle}\,d\mu_t\notag\\
    &=-(\alpha-1)\int_{M_t}{\phi(\phi')^{\alpha-2}\phi''|\nabla r|^2}\,d\mu_t,\label{mono-Vphik}
\end{align}
where we have used \eqref{eq-DeltaPhi} in the first equality. Then Lemma \ref{lem-mono} follows from \eqref{eq-Vphi} and \eqref{mono-Vphik} directly.
\end{proof}

\section{Long time existence and exponential convergence}\label{sec-long}
Throughout this section, we assume that $\mathbf{N}^{n+1}$ is a rotationally symmetric space equipped with the metric: $\bar{g}=dr^2+\phi^2(r)\sigma$ and satisfies the following assumptions:
\begin{equation*}
  \phi,\phi'>0,\quad (\phi')^2-\phi\phi''\geq 0.
\end{equation*}


Since the initial hypersurface $M_0$ is graphical, we can assume that the flow \eqref{flow-WVPF} admits a smooth graphical solution on the maximal time interval $[0,T^{\star})$. Then as discussed in \S\ref{subsection-graph}, the flow \eqref{flow-WVPF} is equivalent to a scalar parabolic PDE on $\mathbb{S}^n$ for the function $\gamma(\cdot,t)$ which satisfies
\begin{equation}\label{eq-Sf}
\partial_t\gamma=(n-\frac{uH}{\phi'})\frac{\omega}{\phi},
\end{equation}
on the time interval $[0,T^{\star})$, where $\omega=\sqrt{1+|D\gamma|^2}$ and $u=\frac{\phi}{\omega}$. Since we have the expression \eqref{eq-WG1} for the Weingarten matrix $\mathcal{W}=\{h_j^i\}$ of $M$, then we have 
\begin{equation}\label{eq-H}
    H=\frac{n\phi'}{\phi\omega}-\frac{1}{\phi\omega}(\sigma^{ik}-\frac{\gamma^i\gamma^k}{\omega^2})\gamma_{ik}
\end{equation}
and hence we can rewritten equation \eqref{eq-Sf} as
\begin{align}
\partial_t \gamma&=\frac{1}{\phi\phi'\omega}(\sigma^{ik}-\frac{\gamma^i\gamma^k}{\omega^2})\gamma_{ik}+n\frac{|D\gamma|^2}{\phi\omega}\notag\\
&=\frac{1}{\phi\phi'\omega}\Delta_0{\gamma}-\frac{\gamma^k\omega_k}{\phi\phi'\omega^2}+n\frac{|D\gamma|^2}{\phi\omega},\label{eq-nodf}\\
&=\text{div}(\frac{D\gamma}{\phi\phi'\omega})+\left(\phi\phi''+(n+1)(\phi')^2\right)\frac{|D\gamma|^2}{\phi(\phi')^2\omega},\label{eq-qlpe}
\end{align}
where we have used the fact that $\omega_k=\frac{\gamma^i\gamma_{ik}}{\omega}$ and $\frac{d\gamma}{d r}=\frac{1}{\phi}$. Then the flow \eqref{flow-WVPF} is equivalent to the following parabolic initial value problem:
\begin{equation}\label{flow-QF}
	\left\{\begin{aligned}
		\partial_t\gamma=&~\text{div}(\frac{D\gamma}{\phi\phi'\omega})+\left((n+1)(\phi')^2+\phi\phi''\right)\frac{|D\gamma|^2}{\phi(\phi')^2\omega}\quad \text{for}\quad  (z,t)\in\mathbb{S}^n\times[0,T^{\star}),\\
		\gamma(\cdot,0)=&~\gamma_0(\cdot).
	\end{aligned}\right.
\end{equation}

Note that the equation \eqref{eq-qlpe} is a divergent quasi-linear parabolic equation, then once the $C^0$ and $C^1$ estimates of $\gamma$ are established, the higher regularity a priori estimates of the solution $\gamma$ follow the standard parabolic theory (cf. \cites{Ural1991,Lieb96} for details).

\subsection{$C^0$ estimate}
In this subsection, we show that $\gamma(\cdot,t)$ is uniformly bounded from above and below.
\begin{prop}\label{prop-C0}
Let $\gamma(\cdot,t)$ be a solution of the initial value problem \eqref{flow-QF}, then for $(z,t)\in\mathbb{S}^n\times[0,T^{\star})$, we have
\begin{align}\label{C0-est} \min_{\theta\in\mathbb{S}^n}\gamma(\theta,0)\leq\gamma(\theta,t)\leq\max_{\theta\in \mathbb{S}^n}\gamma(\theta,0).
    \end{align}
\end{prop}
\begin{proof}
At critical points of $\gamma$, we have
\begin{equation}\label{con-cri}
D \gamma=0,\quad \omega=1.
\end{equation}
Combining  \eqref{con-cri} with \eqref{eq-nodf}, 
we have
\begin{equation}
\partial_t\gamma=\frac{1}{\phi\phi'}\Delta_0 \gamma
\end{equation}
at critical points of $\gamma$. Then by the standard maximum principle, we get the uniform lower and upper bounds \eqref{C0-est} for $\gamma$.
\end{proof}
\begin{rem}
    Since $\gamma$ is an increasing function of $r$, Proposition \ref{prop-C0} implies that the radial function $r(\cdot,t)$ is also attained its maximum and minimum at time $t=0$.
\end{rem}
\begin{rem}
    The assumption $(\phi')^2-\phi\phi''\geq 0$ is no need for getting the $C^0$ estimate of $\gamma$.
\end{rem}

\subsection{$C^1$ estimate} In this subsection, we derive the uniform gradient estimate for $\gamma$. To see this, we first derive the evolution equation for $|D\gamma|^2$.
\begin{lem}\label{lem-eeDgamma}
Let $\gamma(\cdot,t)$ be a solution to the initial value problem \eqref{flow-QF} on the time interval $[0,T^{\star})$. We assume that $|D\gamma|^2$ attains its maximum value at the point $p$, then at the point $p$, there holds
\begin{align}\label{C1-est-eq0}
         \partial_t\vert D\gamma\vert^2=&\frac{1}{\phi\phi'\omega}\left(\sigma^{ij}-\frac{\gamma^i\gamma^j}{\omega^2}\right)(\vert D\gamma\vert^2)_{ij}-\frac{2|D^2\gamma|^2}{\phi\phi'\omega}-2\frac{(\phi')^2+\phi\phi''}{\phi(\phi')^2\omega}{\Delta_0}\gamma\vert D\gamma\vert^2\nonumber\\
        &-2n\frac{\phi'}{\phi\omega}\vert D\gamma\vert^4-2(n-1)\frac{|D\gamma|^2}{\phi\phi'\omega}.
     \end{align}
\end{lem}
\begin{proof}
By the critical point condition, at the point $p$, we have
\begin{align}\label{eq-conC1}
        D\vert D\gamma\vert^2=0,\quad D\omega=0.
    \end{align}
Using equation \eqref{eq-nodf} and critical point condition \eqref{eq-conC1}, we calculate as follows:
\begin{align}
        \partial_t\vert D\gamma\vert^2=2\gamma^i\gamma_{it}=&2\gamma^iD_i\left(\frac{1}{\phi\phi'\omega}\Delta_0\gamma-\frac{\gamma^k\omega_k}{\phi\phi'\omega^2}+n\frac{\vert D\gamma\vert^2}{\phi\omega}\right)\notag\\
        =&2\frac{\gamma^iD_i\Delta_0\gamma}{\phi\phi'\omega}-2\frac{(\phi')^2+\phi\phi''}{\phi(\phi')^2\omega}\Delta_0\gamma\vert D\gamma\vert^2-2\frac{\gamma^i\gamma^k\omega_{ki}}{\phi\phi'\omega^2}-2n\frac{\phi'}{\phi\omega}\vert D\gamma\vert^4.\label{eq-partialDgamma}
    \end{align}
By the Ricci identity on $\mathbb{S}^n$, we have
\begin{align}
(|D\gamma|^2)_{ij}&=2D_j(\gamma^k\gamma_{ki})\notag\\
&=2\gamma_j^k\gamma_{ki}+2\gamma^k\gamma_{kij}\notag\\
&=2\gamma_j^k\gamma_{ki}+2\gamma^k(\gamma_{ijk}+\gamma_k\sigma_{ij}-\gamma_j\sigma_{ik})\notag\\
&=2\gamma_j^k\gamma_{ki}+2\gamma^k\gamma_{ijk}+2|D\gamma|^2\sigma_{ij}-2\gamma_i\gamma_j\label{eq-Dgammaij}
\end{align}
and hence
\begin{align}\label{eq-LDgamma}
\Delta_0{|D\gamma|^2}=2|D^2\gamma|^2+2\gamma^i D_i\Delta_0\gamma+2(n-1)|D\gamma|^2.
\end{align}
Substituting equation \eqref{eq-LDgamma} into equation \eqref{eq-partialDgamma}, we get
\begin{align*}
\partial_t\vert D\gamma\vert^2=&\frac{\Delta_0{|D\gamma|^2}-2|D^2\gamma|^2-2(n-1)|D\gamma|^2}{\phi\phi'\omega}-2\frac{(\phi')^2+\phi\phi''}{\phi(\phi')^2\omega}\Delta_0\gamma\vert D\gamma\vert^2\\
&-2\frac{\gamma^i\gamma^k\omega_{ki}}{\phi\phi'\omega^2}-2n\frac{\phi'}{\phi\omega}\vert D\gamma\vert^4\\
=&\frac{1}{\phi\phi'\omega}\left(\sigma^{ij}-\frac{\gamma^i\gamma^j}{\omega^2}\right)(\vert D\gamma\vert^2)_{ij}-\frac{2|D^2\gamma|^2}{\phi\phi'\omega}-2\frac{(\phi')^2+\phi\phi''}{\phi(\phi')^2\omega}{\Delta_0}\gamma\vert D\gamma\vert^2\\
&-2n\frac{\phi'}{\phi\omega}\vert D\gamma\vert^4-2(n-1)\frac{|D\gamma|^2}{\phi\phi'\omega},
\end{align*}
where in the second equality, we have used the fact that
\begin{equation*}
(|D\gamma|^2)_{ij}=(\omega^2)_{ij}=2\omega\omega_{ij}
\end{equation*}
at the critical point $p$. This completes the proof.
\end{proof}
After the preparation in Proposition \ref{prop-C0} and Lemma \ref{lem-eeDgamma}, we can give the following gradient estimate for $\gamma$.
\begin{prop}\label{prop-C1}
    Let $\gamma(\cdot,t)$ be a solution to the initial value problem \eqref{flow-QF} on the time interval $[0,T^{\star})$, then for any $(z,t)\in\mathbb{S}^n\times[0,T^{\star})$, there exists a constant $\beta>0$, depending only on $n,||\gamma_0||_{C^0(\mathbb{S}^n)}$ and $||\gamma_0||_{C^1(\mathbb{S}^n)}$ such that 
\begin{equation}\label{eq-C1ec}
    |D\gamma|^2(z,t)\leq \max_{z\in\mathbb{S}^n}|D\gamma|^2(z,0)\mathrm{e}^{-\beta t}.
\end{equation}
\end{prop}
\begin{proof}
    By Lemma \ref{lem-eeDgamma}, at the maximum point of $|D\gamma|^2$, we have
    \begin{align}
    \partial_t\vert D\gamma\vert^2\leq&-\frac{2|D^2\gamma|^2}{\phi\phi'\omega}-2\frac{(\phi')^2+\phi\phi''}{\phi(\phi')^2\omega}{\Delta_0}\gamma\vert D\gamma\vert^2-2n\frac{\phi'}{\phi\omega}\vert D\gamma\vert^4-2(n-1)\frac{|D\gamma|^2}{\phi\phi'\omega}\notag\\
    \leq& -\frac{2(\Delta_0{\gamma})^2}{n\phi\phi'\omega}-2\frac{(\phi')^2+\phi\phi''}{\phi(\phi')^2\omega}{\Delta_0}\gamma\vert D\gamma\vert^2-2n\frac{\phi'}{\phi\omega}\vert D\gamma\vert^4-2(n-1)\frac{|D\gamma|^2}{\phi\phi'\omega}\notag\\
    =&-\frac{2}{n\phi\phi'\omega}\left(\Delta_0\gamma+\frac{n}{2}\frac{(\phi')^2+\phi\phi''}{\phi'}\vert D\gamma\vert^2\right)^2+\frac{n}{2}\frac{((\phi')^2+\phi\phi'')^2}{\phi(\phi')^3\omega}\vert D\gamma\vert^4\notag\\
    &-2n\frac{\phi'}{\phi\omega}\vert D\gamma\vert^4-2(n-1)\frac{|D\gamma|^2}{\phi\phi'\omega}\notag\\
    \leq&\frac{n}{2}\frac{((\phi')^2+\phi\phi'')^2}{\phi(\phi')^3\omega}\vert D\gamma\vert^4-2n\frac{\phi'}{\phi\omega}\vert D\gamma\vert^4-2(n-1)\frac{|D\gamma|^2}{\phi\phi'\omega}\label{eq-pDgammaleq}
    \end{align}
Note that 
\begin{align}
&\frac{n}{2}\frac{((\phi')^2+\phi\phi'')^2}{\phi(\phi')^3\omega}\vert D\gamma\vert^4-2n\frac{\phi'}{\phi\omega}\vert D\gamma\vert^4\notag\\
=&\frac{n}{2\phi(\phi')^3\omega}\left(((\phi')^2+\phi\phi'')^2-4(\phi')^4\right)\vert D\gamma\vert^4\notag\\
=&\frac{n}{2\phi(\phi')^3\omega}\left(\phi^2(\phi'')^2+2\phi(\phi')^2\phi''-3(\phi')^4\right)\vert D\gamma\vert^4\notag\\
\leq& 0\label{eq-eleq0}
\end{align}
due to the assumption that $(\phi')^2\geq \phi\phi''$.
Thus we have 
\begin{equation}\label{ineq-partialDgamma}
    \partial_t\vert D\gamma\vert^2\leq -2(n-1)\frac{|D\gamma|^2}{\phi\phi'\omega}\leq 0
\end{equation}
The standard parabolic maximum principle yields that 
\begin{equation}\label{C1-es1}
|D\gamma|^2(z,t)\leq \max_{z\in\mathbb{S}^n}|D\gamma|^2(z,0)
\end{equation}
holds for any $(z,t)\in\mathbb{S}^n\times[0,T^{\star})$. Hence $\omega=\sqrt{1+|D\gamma|^2}$ has a uniform upper bound $C$ which depends only on $||\gamma_0||_{C^1(\mathbb{S}^n)}$. Then by \eqref{ineq-partialDgamma}, we have
\begin{equation*}
\partial_t\vert D\gamma\vert^2\leq -\beta |D\gamma|^2
\end{equation*}
for some $\beta>0$ depending only on $n,||\gamma_0||_{C^0(\mathbb{S}^n)}$ and $||\gamma_0||_{C^1(\mathbb{S}^n)}$. The standard maximum principle then implies \eqref{eq-C1ec}.
\end{proof}
\subsection{Proof of Theorem \ref{Thm-main} and Corollary \ref{Cor-VVphi}} Since \eqref{flow-QF} is a divergent quasi-linear parabolic equation, then by the classical theory of parabolic equation in divergent form(cf. \cites{Ural1991,Lieb96} for details), all higher order regularity estimates of $\gamma$ follows from the $C^0$ and $C^1$ estimates established in Proposition \ref{prop-C0} and \ref{prop-C1}, and a standard continuation argument then shows that $T^{\star}=\infty$, i.e., the flow \eqref{flow-QF}, or equivalently, the flow \eqref{flow-WVPF} exists for all time $t\in[0,\infty)$. The exponential convergence to a slice of $\mathbf{N}^{n+1}$ with radius $r_{\infty}$ follows from the estimate \eqref{eq-C1ec}, where $r_{\infty}$ is determined by the fact that $V_{\phi}(B(r_{\infty}))=V_{\phi}(\Omega_0)$. This completes the proof of Theorem \ref{Thm-main}. Then Corollary \ref{Cor-VVphi} follows from Theorem \ref{Thm-main} and Lemma \ref{lem-mono} immediately.

\section{Preserving of static convexity}\label{Sec-static-preserving}
In this section, we will use the tensor maximum principle to prove that the static convexity is preserved along the flow \eqref{flow-WVPF}, provided that the initial hypersurface $M_0$ is $\varepsilon_0$-close to a slice of $\mathbf{N}^{n+1}$ in the $C^1$ sense for some constant $\varepsilon_0>0$. Precisely, if $M_0$ lies in a bounded domain $B(R)=[r_0,R]\times \mathbb{S}^n$, then $\varepsilon_0$ can be chosen depends only on $n$ and $R$. 

Recall that we say a graphical hypersurface $M_0$ is $\varepsilon_0$-close to a slice of $\mathbf{N}^{n+1}$, if 
\begin{align}
	\vert D\gamma\vert^2\leq \varepsilon_0, \ \text{or equivalently,} \ \frac{u^2}{\phi^2}\geq\frac{1}{1+\varepsilon_0}.
\end{align}
holds everywhere on $M_0$, where $\gamma$ is the variable introduced in \S\ref{subsection-graph}.
Since $\max_{\theta\in \mathbb{S}^n}\vert D\gamma\vert^2(\theta,t)$ is strictly decreasing along the flow \eqref{flow-WVPF} by Proposition \ref{prop-C1}, the upper bound $\varepsilon_0$ is valid as long as the flow exists. We emphasize that the preserving of static convexity is essential in proving the monotonicity of the weighted area along the flow \eqref{flow-WVPF}. 

For convenience of the readers, we recall the tensor maximum principle, which was first proved by Hamilton \cite{Ham-1982} and was generalized by Andrews \cite{Ben-2007}.
\begin{thm}[\cite{Ben-2007}]\label{tensormax}
	Let $S_{ij}$ be a smooth time-varying symmetric tensor field on a compact manifold $M$ (possibly with boundary), satisfying 
	\begin{align*}
		\frac{\partial}{\partial t}S_{ij}=a^{k\ell}\nabla_k\nabla_{\ell}S_{ij}+b^k\nabla_kS_{ij} +N_{ij},
	\end{align*}
	where $a^{k\ell}$ and $b$ are smooth, $\nabla$ is a (possibly time-dependent) smooth symmetric connection, and $a^{k\ell}$ is positive definite everywhere. Suppose that
	\begin{align}\label{con tensormax}
		N_{ij}v^iv^j+\sup_{\Lambda}2a^{k\ell}(2\Lambda_k^p\nabla_{\ell}S_{ip}v^i-\Lambda_k^p\Lambda_{\ell}^qS_{pq})\geq 0,
	\end{align}
	where $S_{ij}\geq 0$ and $S_{ij}v^j=0$. If $S_{ij}$ is positive definite everywhere on $M$ at time $t=0$ and on $\partial M$ for $0\leq t\leq T$, then it is positive on $M\times[0,T]$.
\end{thm}

In our case, we take $S_{ij}=h_{ij}-\frac{u\phi''}{\phi\phi'}g_{ij}$, we will show that $S_{ij}\geq 0$ is preserved along the flow \eqref{flow-WVPF}, provided that the initial hypersurface $M_0$ is $\varepsilon_0$-close to a slice of $\mathbf{N}^{n+1}$ in the $C^1$ sense for some $\epsilon_0>0$.

\subsection{Evolution equation of $S_{ij}$} First of all, we derive the evolution equation of $S_{ij}$ along the flow \eqref{flow-WVPF} in the static rotationally symmetric space $\mathbf{N}^{n+1}$. For simplicity of notations, the following calculations are carried on under a local orthonormal basis and we use the symbol $\langle\cdot,\cdot\rangle$ to denote the inner product during the calculations without ambiguity.

\begin{lem}\label{evolution-h}
Along the flow \eqref{flow-WVPF}, the second fundamental form $h_{ij}$ evolves by
\begin{align}
	\partial_t h_{ij}=&\frac{u}{\phi'}\Delta_g h_{ij}+\frac{H}{\phi'}\langle\phi\partial_r,\nabla h_{ij}\rangle+\nabla_iH\nabla_j\frac{u}{\phi'}+\nabla_jH\nabla_i\frac{u}{\phi'}\nonumber\\
 &-\frac{H\phi''}{\phi'}\left(\langle\partial_r,e_i\rangle\nabla_j\frac{u}{\phi'}+\langle\partial_r,e_j\rangle\nabla_i\frac{u}{\phi'}\right)+(n-3\frac{uH}{\phi'})(h^2){_{ij}}\nonumber\\
 &+\frac{u}{\phi'}\vert A\vert^2h_{ij}+\frac{u^2H\phi''}{\phi(\phi')^2}h_{ij}-2H\frac{u\phi''}{\phi\phi'}g_{ij}+Hh_{ij}\nonumber\\
   &-\frac{uH}{\phi(\phi')^2}(\phi\phi'''-\phi'\phi'')\langle\partial_r,e_i\rangle\langle\partial_r,e_j\rangle+\frac{u}{\phi'}h_{ij}\overline{Ric}(\nu,\nu)\nonumber\\
 &+\frac{u}{\phi'}\sum_{k=1}^n\left(\overline{\nabla}_k\bar{R}_{\nu jki}+\overline{\nabla}_i\bar{R}_{\nu kkj}\right)+2\frac{u}{\phi'}\sum_{k,\ell=1}^nh_{k\ell}\bar{R}_{\ell jki}-(n+\frac{uH}{\phi'})\bar{R}_{i \nu j \nu }\nonumber\\
   &+\frac{u}{\phi'}\sum_{k,\ell=1}^n\left(h_{\ell j}\bar{R}_{\ell kki}+h_{i\ell}\bar{R}_{\ell kkj}\right)\label{eveq-h2}
\end{align}
\end{lem}
\begin{proof}
    By (\ref{eveq-h1}), along the flow (\ref{flow-WVPF}) we have
    \begin{align}
        \partial_th_{ij}=&-\nabla_i\nabla_j(n-\frac{uH}{\phi'})+(n-\frac{uH}{\phi'})(h^2){_{ij}}-(n-\frac{uH}{\phi'})\bar{R}_{i\nu j\nu}\nonumber\\
        =&\nabla_i\left(\frac{u}{\phi'}\nabla_jH+H\nabla_j\frac{u}{\phi'}\right)+(n-\frac{uH}{\phi'})(h^2){_{ij}}-(n-\frac{uH}{\phi'})\bar{R}_{i\nu j\nu}\nonumber\\
        =&\frac{u}{\phi'}\nabla_i\nabla_jH+\nabla_iH\nabla_j\frac{u}{\phi'}+\nabla_jH\nabla_i\frac{u}{\phi'}+H\nabla_i\nabla_j\frac{u}{\phi'}\nonumber\\
        &+(n-\frac{uH}{\phi'})(h^2){_{ij}}-(n-\frac{uH}{\phi'})\bar{R}_{i\nu j\nu}.\label{eqh1-1}
    \end{align}
    First, we derive
    \begin{align}
       H\nabla_i\nabla_j\frac{u}{\phi'}=&H\nabla_i\left(\frac{\nabla_ju}{\phi'}-\frac{u\phi''}{\phi(\phi')^2}\langle\phi\partial_r,e_j\rangle\right)\nonumber\\
       =&\frac{H}{\phi'}\nabla_i\nabla_ju-\frac{H\phi''}{(\phi')^2}(\langle\partial_r,e_i\rangle\nabla_ju+\langle\partial_r,e_j\rangle\nabla_iu)-\frac{uH\phi''}{\phi(\phi')^2}(\phi'g_{ij}-uh_{ij})\nonumber\\
       &-\frac{uH}{\phi(\phi')^3}(\phi\phi'\phi'''-(\phi')^2\phi''-2\phi(\phi'')^2)\langle\partial_r,e_i\rangle\langle\partial_r,e_j\rangle\nonumber\\
       =&\frac{H}{\phi'}\nabla_i\nabla_ju-\frac{H\phi''}{\phi'}\left(\langle\partial_r,e_i\rangle\nabla_j\frac{u}{\phi'}+\langle\partial_r,e_j\rangle\nabla_i\frac{u}{\phi'}\right)\nonumber\\
       &-\frac{uH\phi''}{\phi(\phi')^2}(\phi'g_{ij}-uh_{ij})-\frac{uH}{\phi(\phi')^3}(\phi\phi'\phi'''-(\phi')^2\phi'')\langle\partial_r,e_i\rangle\langle\partial_r,e_j\rangle,\label{eqh1-2}
    \end{align}
    where in the last equality we used the formula $\nabla_iu=\nabla_i\langle\phi\partial_r,\nu\rangle=\sum_{k=1}^nh_{ik}\langle\phi\partial_r,e_k\rangle$. Then, by Codazzi equation, a directly calculation gives
    \begin{align}\label{eqh1-3}
        \frac{H}{\phi'}\nabla_i\nabla_ju=&\frac{H}{\phi'}\sum_{k=1}^nh_{jk,i}\langle\phi\partial_r,e_k\rangle+Hh_{ij}-\frac{uH}{\phi'}(h^2)_{ij}\nonumber\\
        =&\frac{H}{\phi'}\langle\phi\partial_r,\nabla h_{ij}\rangle+\frac{H}{\phi'}\sum_{k=1}^n\bar{R}_{\nu jki}\langle\phi\partial_r,e_k\rangle+Hh_{ij}-\frac{uH}{\phi'}(h^2){_{ij}}\nonumber\\
        =&\frac{H}{\phi'}\langle\phi\partial_r,\nabla h_{ij}\rangle-\frac{uH}{\phi'}\bar{R}_{\nu j\nu i}+H\frac{\phi}{\phi'}\bar{R}_{\nu j r i}+Hh_{ij}-\frac{uH}{\phi'}(h^2)_{ij}.
    \end{align}
    By (\ref{eq-gnR}), we have
    \begin{align}\label{eqh1-4}
        \bar{R}_{\nu j r i}=-\frac{u\phi''}{\phi^2}g_{ij},
    \end{align}
    Substituting (\ref{eqh1-4}) into (\ref{eqh1-3}), we get
    \begin{align}\label{eqh1-5}
        \frac{H}{\phi'}\nabla_i\nabla_ju=&\frac{H}{\phi'}\langle\phi\partial_r,\nabla h_{ij}\rangle-\frac{uH}{\phi'}\bar{R}_{\nu j\nu i}-H\frac{u\phi''}{\phi\phi'}g_{ij}+Hh_{ij}-\frac{uH}{\phi'}(h^2)_{ij}.
    \end{align}
Combining (\ref{eqh1-5}) with (\ref{eqh1-2}), we have
\begin{align}
   H\nabla_i\nabla_j\frac{u}{\phi'}=&\frac{H}{\phi'}\langle\phi\partial_r,\nabla h_{ij}\rangle-\frac{uH}{\phi'}\bar{R}_{\nu j\nu i}+\frac{u^2H\phi''}{\phi(\phi')^2}h_{ij}-2\frac{uH\phi''}{\phi\phi'}g_{ij}\nonumber\\
   &+Hh_{ij}-\frac{uH}{\phi'}(h^2)_{ij}-\frac{H\phi''}{\phi'}\left(\langle\partial_r,e_i\rangle\nabla_j\frac{u}{\phi'}+\langle\partial_r,e_j\rangle\nabla_i\frac{u}{\phi'}\right)\nonumber\\
       &-\frac{uH}{\phi(\phi')^2}(\phi\phi'''-\phi'\phi'')\langle\partial_r,e_i\rangle\langle\partial_r,e_j\rangle.\label{eqh1-6}
\end{align}
Next, by Codazzi equation and Ricci identity we have
    \begin{align}
        \nabla_i\nabla_jH=&\sum_{k=1}^nh_{kk,ji}=\sum_{k=1}^n\nabla_i(h_{kj,k}+\bar{R}_{\nu kkj})
        =\sum_{k=1}^n(h_{kj,ki}+\nabla_i\bar{R}_{\nu kkj})\nonumber\\
        =&\sum_{k=1}^n\Big[h_{jk,ik}+\sum_{\ell=1}^n(h_{\ell j}R_{\ell kki}+h_{k\ell}R_{\ell jki})+\overline{\nabla}_i\bar{R}_{\nu kkj}+\sum_{\ell=1}^nh_{i\ell}\bar{R}_{\ell kkj}\nonumber\\
        &-h_{ik}\bar{R}_{\nu k\nu j}-h_{ij}\bar{R}_{\nu kk\nu}\Big]\nonumber\\
        =&\sum_{k=1}^n (h_{ij,kk}+\nabla_k\bar{R}_{\nu jki})+\sum_{k,\ell=1}^nh_{\ell j}\left(\bar{R}_{\ell kki}+h_{k\ell}h_{ki}-h_{kk}h_{\ell i}\right)\nonumber\\
        &+\sum_{k,\ell=1}^nh_{k\ell}\left(\bar{R}_{\ell jki}+h_{k\ell}h_{ij}-h_{\ell i}h_{kj}\right)+\sum_{k=1}^n\overline{\nabla}_i\bar{R}_{\nu kkj}-\sum_{\ell=1}^nh_{i\ell}\overline{Ric}(e_\ell, e_j)\nonumber\\
        &+\sum_{\ell=1}^nh_{i\ell}\bar{R}_{\ell\nu j\nu}-\sum_{k=1}^nh_{ik}\bar{R}_{\nu k\nu j}+h_{ij}\overline{Ric}(\nu,\nu)\nonumber\\
        =&\Delta_g h_{ij}+\sum_{k=1}^n(\overline{\nabla}_k\bar{R}_{\nu jki}+\overline{\nabla}_i\bar{R}_{\nu kkj})+2\sum_{k,\ell=1}^nh_{k\ell}\bar{R}_{\ell jki}-H\bar{R}_{\nu j\nu i}\nonumber\\
        &+\sum_{k,\ell=1}^n(h_{\ell j}\bar{R}_{\ell kki}+h_{i\ell}\bar{R}_{\ell kkj})
        +h_{ij}\overline{Ric}(\nu,\nu)-H(h^2)_{ij}+\vert A\vert^2h_{ij}.\label{eqh1-7}
    \end{align}
Now, substituting (\ref{eqh1-6}) and (\ref{eqh1-7}) into (\ref{eqh1-1}) we obtain (\ref{eveq-h2}).
\end{proof}
\begin{lem}\label{evolution-sf}
    Along the flow (\ref{flow-WVPF}), the function $\frac{u\phi''}{\phi\phi'}$ evloves by
    \begin{align}\label{eqsf}
        \partial_t\frac{u\phi''}{\phi\phi'}=&\frac{u}{\phi'}\Delta_g(\frac{u\phi''}{\phi\phi'})+\frac{H}{\phi'}\langle\phi\partial_r,\nabla \frac{u\phi''}{\phi\phi'}\rangle-2u\langle\nabla\frac{u}{\phi'},\nabla\frac{\phi''}{\phi\phi'}\rangle+n\frac{\phi''}{\phi}-2\frac{u\phi''}{\phi\phi'}H+\frac{u^2\phi''}{\phi(\phi')^2}\vert A\vert^2\nonumber\\
		&+\frac{u^2\phi''}{\phi(\phi')^2}\overline{Ric}(\nu,\nu)+n(\frac{u\phi''}{\phi\phi'})^2-\frac{u^2}{\phi'}\langle\phi\partial_r,\nabla\frac{\phi\phi'\phi'''-\phi(\phi'')^2-(\phi')^2\phi''}{\phi^3(\phi')^2}\rangle \nonumber\\
  &-uH\frac{\phi\phi'''-\phi'\phi''}{\phi(\phi')^2}(1-\frac{u^2}{\phi^2})-2\frac{u^2\phi''}{(\phi')^2}\langle\partial_r,\nabla\frac{\phi''}{\phi\phi'}\rangle.
    \end{align}
\end{lem}
\begin{proof}
  Recall (\ref{flow-SWVPF}), we have
\begin{align*}
	\partial_tr=\phi\partial_t\gamma=(n-\frac{uH}{\phi'})\frac{u}{\phi}.
\end{align*}
Thus we get
\begin{align}\label{eqsf-1}
    \partial_t\frac{\phi''}{\phi\phi'}=&\frac{\phi\phi'\phi'''-\phi(\phi'')^2-(\phi')^2\phi''}{\phi^2(\phi')^2}\partial_tr\nonumber\\
    =&\frac{\phi\phi'\phi'''-\phi(\phi'')^2-(\phi')^2\phi''}{\phi^3(\phi')^2}(n-\frac{uH}{\phi'})u.
\end{align}
On the other hand,
\begin{align*}
	\nabla_i\frac{\phi''}{\phi\phi'}
	=&\frac{\phi\phi'\phi'''-\phi(\phi'')^2-(\phi')^2\phi''}{\phi^3(\phi')^2}\langle\phi\partial_r,e_i\rangle,
\end{align*}
then we have
\begin{align}\label{eqsf-2}
	\Delta_g\frac{\phi''}{\phi\phi'}=&\frac{\phi\phi'\phi'''-\phi(\phi'')^2-(\phi')^2\phi''}{\phi^3(\phi')^2}(n\phi'-uH)\nonumber\\
	&+\langle\phi\partial_r,\nabla\frac{\phi\phi'\phi'''-\phi(\phi'')^2-(\phi')^2\phi''}{\phi^3(\phi')^2}\rangle.
\end{align}
Hence, combining (\ref{eqsf-1}) with (\ref{eqsf-2}) we obtain 
\begin{align}\label{eqsf-3}
	\partial_t\frac{\phi''}{\phi\phi'}=\frac{u}{\phi'}\Delta_g\frac{\phi''}{\phi\phi'}-\frac{u}{\phi'}\langle\phi\partial_r,\nabla\frac{\phi\phi'\phi'''-\phi(\phi'')^2-(\phi')^2\phi''}{\phi^3(\phi')^2}\rangle.
\end{align}

Next, we derive the evolution equation of the support function $u$ along the flow (\ref{flow-WVPF}).

By the definition of $u$ and the evolution equation (\ref{eveq-normal}) of the unit outward normal $\nu$, we have
\begin{align}\label{eqsf-4}
	\partial_tu=\partial_t\langle\phi\partial_r,\nu\rangle
	=n\phi'-uH+\langle\phi\partial_r,\nabla\frac{uH}{\phi'}\rangle.
\end{align}
Meanwhile, since
\begin{align}\label{eqsf-9}
	\nabla_iu=\sum_{k=1}^nh_{ik}\langle\phi\partial_r,e_k\rangle,
\end{align}
then by Codazzi equation we get
\begin{align}\label{eqsf-5}
	\Delta_g u=&\sum_{k=1}^nh_{ik,i}\langle\phi\partial_r,e_k\rangle+\phi'H-\vert A\vert^2u\nonumber\\
	=&\langle\phi\partial_r,\nabla H\rangle+\sum_{k=1}^n\bar{R}_{\nu iki}\langle\phi\partial_r,e_k\rangle+\phi'H-\vert A\vert^2u\nonumber\\
	=&\langle\phi\partial_r,\nabla H\rangle -u\overline{Ric}(\nu,\nu)+\phi\overline{Ric}(\partial_r,\nu)+\phi'H-\vert A\vert^2u\nonumber\\
 =&\langle\phi\partial_r,\nabla H\rangle -u\overline{Ric}(\nu,\nu)-n\frac{u\phi''}{\phi}+\phi'H-\vert A\vert^2u,
\end{align}
where in the last equality we used the expression (\ref{eq-RicciT}) of the Ricci curvature tensor.
Therefore, using (\ref{eqsf-4}) and (\ref{eqsf-5}) we obtain
\begin{align}\label{eveq-u}
	\partial_tu=&\frac{u}{\phi'}\Delta_g u+H\langle\phi\partial_r,\nabla\frac{u}{\phi'}\rangle+n\phi'-2uH+\frac{u^2}{\phi'}\vert A\vert^2+n\frac{u^2\phi''}{\phi\phi'}+\frac{u^2}{\phi'}\overline{Ric}(\nu,\nu).
\end{align}
Hence, combining (\ref{eqsf-3}) with (\ref{eveq-u}) gives
\begin{align}\label{eqsf-6}
		\partial_t\frac{u\phi''}{\phi\phi'}=&u\partial_t\frac{\phi''}{\phi\phi'}+\frac{\phi''}{\phi\phi'}\partial_tu\nonumber\\
  =&u\left(\frac{u}{\phi'}\Delta_g\frac{\phi''}{\phi\phi'}-\frac{u}{\phi'}\langle\phi\partial_r,\nabla\frac{\phi\phi'\phi'''-\phi(\phi'')^2-(\phi')^2\phi''}{\phi^3(\phi')^2}\rangle\right)\nonumber\\
		&+\frac{\phi''}{\phi\phi'}\Big(\frac{u}{\phi'}\Delta_g u+H\langle\phi\partial_r,\nabla\frac{u}{\phi'}\rangle+n\phi'-2uH+\frac{u^2}{\phi'}\vert A\vert^2+n\frac{u^2\phi''}{\phi\phi'}+\frac{u^2}{\phi'}\overline{Ric}(\nu,\nu)\Big)\nonumber\\
		=&\frac{u}{\phi'}\Delta_g(\frac{u\phi''}{\phi\phi'})-2\frac{u}{\phi'}\langle\nabla u,\nabla\frac{\phi''}{\phi\phi'}\rangle+\frac{H\phi''}{\phi\phi'}\langle\phi\partial_r,\nabla \frac{u}{\phi'}\rangle+n\frac{\phi''}{\phi}-2\frac{u\phi''}{\phi\phi'}H+\frac{u^2\phi''}{\phi(\phi')^2}\vert A\vert^2\nonumber\\
		&+\frac{u^2\phi''}{\phi(\phi')^2}\overline{Ric}(\nu,\nu)+n(\frac{u\phi''}{\phi\phi'})^2-\frac{u^2}{\phi'}\langle\phi\partial_r,\nabla\frac{\phi\phi'\phi'''-\phi(\phi'')^2-(\phi')^2\phi''}{\phi^3(\phi')^2}\rangle.
\end{align}
Note that
\begin{align}\label{eqsf-7}
    \frac{H\phi''}{\phi\phi'}\langle\phi\partial_r,\nabla \frac{u}{\phi'}\rangle=&\frac{H}{\phi'}\langle\phi\partial_r,\nabla \frac{u\phi''}{\phi\phi'}\rangle-\frac{uH}{(\phi')^2}\langle\phi\partial_r,\nabla\frac{\phi''}{\phi}\rangle\nonumber\\
    =&\frac{H}{\phi'}\langle\phi\partial_r,\nabla \frac{u\phi''}{\phi\phi'}\rangle-uH\frac{\phi\phi'''-\phi'\phi''}{\phi(\phi')^2}(1-\frac{u^2}{\phi^2})
\end{align}
and
\begin{align}\label{eqsf-8}
    2\frac{u}{\phi'}\langle\nabla u,\nabla\frac{\phi''}{\phi\phi'}\rangle=&2u\langle\nabla\frac{u}{\phi'},\nabla\frac{\phi''}{\phi\phi'}\rangle+2\frac{u^2\phi''}{(\phi')^2}\langle\partial_r,\nabla\frac{\phi''}{\phi\phi'}\rangle.
\end{align}
Substituting (\ref{eqsf-7}) and (\ref{eqsf-8}) into (\ref{eqsf-6}), we finally get (\ref{eqsf}).
\end{proof}
Applying Lemma \ref{evolution-h} and \ref{evolution-sf}, we are prepared to obtain the evolution equation of the tensor $S_{ij}$ along the flow (\ref{flow-WVPF}) now.
\begin{prop}
    Along the flow (\ref{flow-WVPF}), we have the following evolution equation of $S_{ij}$:
    \begin{align}
        \partial_tS_{ij}=&\frac{u}{\phi'}\Delta_g S_{ij}+\nabla_iH\nabla_j\frac{u}{\phi'}+\nabla_jH\nabla_i\frac{u}{\phi'}+\frac{H}{\phi'}\langle\phi\partial_r,\nabla S_{ij}\rangle+\frac{u}{\phi'}S_{ij}\overline{Ric}(\nu,\nu)+3\frac{u^2 H\phi''}{\phi(\phi')^2}S_{ij}\nonumber\\
	&+(n-3\frac{uH}{\phi'})((S^2)_{ij}+2\frac{\phi''u}{\phi\phi'}S_{ij})+\frac{u}{\phi'}\vert A\vert^2S_{ij}-\frac{H\phi''}{\phi'}\left(\langle\partial_r,e_i\rangle\nabla_j\frac{u}{\phi'}+\langle\partial_r,e_j\rangle\nabla_i\frac{u}{\phi'}\right)\nonumber\\
   &+Hh_{ij}-2n\frac{u\phi''}{\phi\phi'}h_{ij}-\frac{uH}{\phi(\phi')^2}(\phi\phi'''-\phi'\phi'')\langle\partial_r,e_i\rangle\langle\partial_r,e_j\rangle\nonumber\\
       &+\Big[2u\langle\nabla\frac{u}{\phi'},\nabla\frac{\phi''}{\phi\phi'}\rangle-n\frac{\phi''}{\phi}+\frac{u^2}{\phi'}\langle\phi\partial_r,\nabla\frac{\phi\phi'\phi'''-\phi(\phi'')^2-(\phi')^2\phi''}{\phi^3(\phi')^2}\rangle \nonumber\\
  &+uH\frac{\phi\phi'''-\phi'\phi''}{\phi(\phi')^2}(1-\frac{u^2}{\phi^2})+2\frac{u^2\phi''}{(\phi')^2}\langle\partial_r,\nabla\frac{\phi''}{\phi\phi'}\rangle\Big]g_{ij}+\frac{u}{\phi'}\sum_{k=1}^n\left(\overline{\nabla}_k\bar{R}_{\nu jki}+\overline{\nabla}_i\bar{R}_{\nu kkj}\right)\nonumber\\
 &+2\frac{u}{\phi'}\sum_{k,\ell=1}^nh_{k\ell}\bar{R}_{\ell jki}-(n+\frac{uH}{\phi'})\bar{R}_{\nu j\nu i}+\frac{u}{\phi'}\sum_{k,\ell=1}^n\left(h_{\ell j}\bar{R}_{\ell kki}+h_{i\ell}\bar{R}_{\ell kkj}\right).\label{eveq-S}
    \end{align}
\end{prop}
\begin{proof}
By the definition of $S_{ij}$,  we have
\begin{align}\label{eq-evoSij}
   \partial_tS_{ij}=&\partial_th_{ij}-\partial_t(\frac{u\phi''}{\phi\phi'})g_{ij}-\frac{u\phi''}{\phi\phi'}\partial_tg_{ij}. 
\end{align}
Substituting (\ref{eveq-metric}), (\ref{eveq-h2}) and (\ref{eqsf}) into \eqref{eq-evoSij}, a direct calculation yields (\ref{eveq-S}).

\end{proof}

\subsection{Proof of Theorem \ref{preservingofconvex}}
The static convexity of the flow hypersurface $M_t$ is equivalent to $S_{ij}\geq0$ holds everywhere on $M_t$, 
it's time to apply the tensor maximum principle in Theorem \ref{tensormax}.

Let $(x_0,t_0)$ be the point where $S_{ij}$ has a null vector $v$. We choose normal coordinates around $(x_0,t_0)$ such that $h_{ij}=k_i\delta_{ij}$ and $g_{ij}=\delta_{ij}$ at this point, where $\kappa_1<\kappa_2<\cdots<\kappa_n$ are the principal curvatures. By continuity, it's reasonable to assume that the principal curvatures are mutually distinct and in increasing order at $(x_0,t_0)$. The null vector condition implies that $v=e_1$ and $S_{11}=\kappa_1-\frac{u\phi''}{\phi\phi'}=0$ at $(x_0,t_0)$. The terms involving $S_{ij}$ and $(S^2)_{ij}$ satisfy the null vector condition and can be ignored. Moreover, by (\ref{eqsf-9}) we have
\begin{align}\label{preconv-1}
    \nabla_i\frac{u}{\phi'}=\frac{\phi}{\phi'}(\kappa_i-\frac{u\phi''}{\phi\phi'})\langle\partial_r,e_i\rangle=\frac{\phi}{\phi'}S_{ii}\langle\partial_r,e_i\rangle.
\end{align}
Since $\kappa_1=\frac{u\phi''}{\phi\phi'}$, this implies $\nabla_1\frac{u}{\phi'}=0$ at $(x_0,t_0)$. Let $N_{11}$ be the remaining terms in the RHS of (\ref{eveq-S}), we now have 
\begin{align}\label{eq-cN11}
    N_{11}=&H\frac{u\phi''}{\phi\phi'}-2n(\frac{u\phi''}{\phi\phi'})^2+\frac{uH}{\phi(\phi')^2}(\phi\phi'''-\phi'\phi'')(1-\frac{u^2}{\phi^2}-\langle\partial_r,e_1\rangle^2)\nonumber\\
       &+2u\langle\nabla\frac{u}{\phi'},\nabla\frac{\phi''}{\phi\phi'}\rangle-n\frac{\phi''}{\phi}+\frac{u^2}{\phi'}\langle\phi\partial_r,\nabla\frac{\phi\phi'\phi'''-\phi(\phi'')^2-(\phi')^2\phi''}{\phi^3(\phi')^2}\rangle \nonumber\\
  &+2\frac{u^2\phi''}{(\phi')^2}\langle\partial_r,\nabla\frac{\phi''}{\phi\phi'}\rangle+\frac{u}{\phi'}\sum_{k=1}^{n}{(\overline{\nabla}_{k}\bar{R}_{\nu 1 k 1}+\overline{\nabla}_1\bar{R}_{\nu k k 1})}+2\frac{u}{\phi'}\sum_{k=1}^{n}{S_{k k}\bar{R}_{k 1 k 1}}\nonumber\\
  &-(n+\frac{uH}{\phi'})\bar{R}_{\nu 1 \nu 1}.
\end{align}
To get the concrete expression of $N_{11}$, we need to calculate the items with regard to Riemaniann curvature tensors and their derivatives carefully. Firstly, we have the following proposition.
\begin{prop}\label{prop-RNR}
	At the point $(x_0,t_0)$ where the tensor maximum principle applies, for any $k\geq 2$ we have:
	\begin{align}
		\overline{\nabla}_{k}\bar{R}_{\nu 1 k 1}=&u\phi\langle \overline{\nabla}{\frac{(\phi')^2-\phi\phi''-1}{\phi^4}},e_{k}\rangle\langle \partial_r, e_{k}\rangle+u\phi'\frac{(\phi')^2-\phi\phi''-1}{\phi^4},\label{eq-nRl}\\
		\overline{\nabla}_{1}\bar{R}_{\nu k k 1}=&-u\phi\langle\overline{\nabla}{\frac{(\phi')^2-\phi\phi''-1}{\phi^4}},e_1\rangle\langle{\partial_r, e_1}\rangle-u\phi'\frac{(\phi')^2-\phi\phi''-1}{\phi^4},\label{eq-nRll}\\
		\overline{R}_{k 1 k 1}=&-\frac{(\phi')^2-1}{\phi^2}+\frac{(\phi')^2-\phi\phi''-1}{\phi^2}(\langle \partial_r, e_1\rangle^2+\langle \partial_r, e_{k}\rangle^2),\label{eq-Rl1}\\
		\overline{R}_{\nu 1 \nu 1}=&-\frac{(\phi')^2-1}{\phi^2}+\frac{(\phi')^2-\phi\phi''-1}{\phi^2}(\langle \partial_r, e_1\rangle^2+\frac{u^2}{\phi^2}).\label{eq-Rnu1}
	\end{align}	
\end{prop}
\begin{proof}
	Note that equations \eqref{eq-Rl1} and \eqref{eq-Rnu1} follows from equation \eqref{eq-gnR} directly and equation \eqref{eq-nRll} can be deduced from equation \eqref{eq-nRl} with swapping the items $e_1$ and $e_{k}$. Hence, in the following, we just show the detail of the calculation of $\overline{\nabla}_{k}\bar{R}_{\nu 1 k 1}$.
	
   For any fixed $k\geq 2$, we have 
	\begin{align}\label{eq-calnR}
		\overline{\nabla}_{k}\bar{R}_{\nu 1 k 1}=&\overline{\nabla}\bar{R}(\nu,e_1,e_{k},e_1;e_{k})\notag\\
		=&e_{k}(\bar{R}(\nu,e_1,e_{k},e_1))-\bar{R}(\overline{\nabla}_{e_{k}}{\nu},e_1,e_{k},e_1)-\bar{R}(\nu,\overline{\nabla}_{e_{k}}e_1,e_{k},e_1)\notag\\
		&-\bar{R}(\nu,e_1,\overline{\nabla}_{e_{k}}e_{k},e_1)-\bar{R}(\nu,e_1,e_{k},\overline{\nabla}_{e_{k}}e_1)\nonumber\\
  =&e_{k}(\bar{R}(\nu,e_1,e_{k},e_1))-h_{kk}\bar{R}(e_k,e_1,e_{k},e_1)+h_{kk}\bar{R}(\nu,e_1,\nu,e_1).
	\end{align}
By the equation \eqref{eq-gnR}, we have
\begin{align*}
    \bar{R}(\nu,e_1,e_{k},e_1)=&u\frac{(\phi')^2-\phi\phi''-1}{\phi^3}\langle \partial_r, e_{k}\rangle,\\
    \bar{R}(\nu,e_1,\nu,e_1)=&\frac{(\phi')^2-\phi\phi''-1}{\phi^2}(\frac{u^2}{\phi^2}+\langle \partial_r, e_{1}\rangle^2)-\frac{(\phi')^2-1}{\phi^2}.
\end{align*}
Then we can calculate
\begin{align*}
	e_{k}(\bar{R}(\nu,e_1,e_{k},e_1))=&h_{kk}\frac{(\phi')^2-\phi\phi''-1}{\phi^2}\langle{\partial_r,e_{k}}\rangle^2+u\phi\langle{\overline{\nabla}{\frac{(\phi')^2-\phi\phi''-1}{\phi^4}},e_{k}}\rangle\langle{\partial_r, e_{k}}\rangle\\
	&+u\frac{(\phi')^2-\phi\phi''-1}{\phi^4}(\phi'-uh_{kk}).
\end{align*}
Substituting the above equations and \eqref{eq-Rl1} into equation \eqref{eq-calnR} gives equation \eqref{eq-nRl}.
\end{proof}
Finally, we give the concrete expression of $N_{11}$ as follows using the static condition (\ref{eq-sTa}).
\begin{prop}\label{Thm-generalN}
	At the point $(x_0,t_0)$ where the tensor maximum principle applies, we have 
     \begin{align}\label{eq-N11}
	N_{11}=&n\frac{(\phi')^2-\phi\phi''-1}{\phi^2}\Bigg[\left(1+\frac{uH}{\phi'}-(n+3)\frac{u^2}{\phi^2}\right)(1-\frac{u^2}{\phi^2}-\langle{\partial_r,e_1}\rangle^2)\nonumber\\
	&-\frac{2u}{n\phi'}trS(1-\langle \partial_r, e_1\rangle^2)
    +2\frac{u}{\phi'}\sum_{k=1}^nS_{kk}\langle\partial_r,e_k\rangle^2\Bigg]-2\frac{u(\phi'')^2}{(\phi')^3}\sum_{k=1}^nS_{kk}\langle\partial_r,e_k\rangle^2,
\end{align}
where $trS:=H-n\frac{u\phi''}{\phi\phi'}$ is the trace of $S_{ij}$.
\end{prop}
\begin{proof}
	By Proposition \ref{prop-RNR}, we calculate as follows:
	\begin{align*}
		&\sum_{k=1}^{n}{(\overline{\nabla}_{k}\bar{R}_{\nu 1 k 1}+\overline{\nabla}_1\bar{R}_{\nu k k 1})}\\
		=&u\phi\sum_{k=1}^{n}\left(\langle{\overline{\nabla}{\frac{(\phi')^2-\phi\phi''-1}{\phi^4}},e_{k}}\rangle\langle{\partial_r,e_{k}}\rangle-\langle{\overline{\nabla}{\frac{(\phi')^2-\phi\phi''-1}{\phi^4}},e_1}\rangle\langle{\partial_r,e_1}\rangle\right)\\
		=&u\phi\left(\langle{\overline{\nabla}{\frac{(\phi')^2-\phi\phi''-1}{\phi^4}},\partial_r-\frac{u}{\phi}\nu}\rangle-n\langle{\overline{\nabla}{\frac{(\phi')^2-\phi\phi''-1}{\phi^4}},e_1}\rangle\langle{\partial_r,e_1}\rangle\right).
	\end{align*}
A direct calculation leads to
\begin{align}\label{eq-phi''}
	\overline{\nabla}{\frac{(\phi')^2-\phi\phi''-1}{\phi^4}}=\left(5\frac{\phi'\phi''}{\phi^4}-\frac{\phi'''}{\phi^3}-4\frac{\phi'((\phi')^2-1)}{\phi^5}\right)\partial_r.
\end{align}
Then 
\begin{align}
	&\sum_{k=1}^{n}{(\overline{\nabla}_{k}\bar{R}_{\nu 1 k 1}+\overline{\nabla}_1\bar{R}_{\nu k k 1})}\notag\\
	=&u\phi\left(5\frac{\phi'\phi''}{\phi^4}-\frac{\phi'''}{\phi^3}-4\frac{\phi'((\phi')^2-1)}{\phi^5}\right)\left(1-\frac{u^2}{\phi^2}-n\langle\partial_r,e_1\rangle^2\right).\label{eq-snR}
\end{align}
Similarly, 
\begin{align}
	\sum_{k=1}^{n}{S_{kk}\bar{R}_{k 1 k 1}}=&\frac{(\phi')^2-\phi\phi''-1}{\phi^2}trS\langle\partial_r,e_1\rangle^2+\frac{(\phi')^2-\phi\phi''-1}{\phi^2}\sum_{k=1}^n{S_{kk}\langle\partial_r,e_{k}\rangle^2}\nonumber\\
 &-\frac{(\phi')^2-1}{\phi^2}trS.\label{eq-shR}
\end{align}
Combining equations \eqref{eq-cN11}, \eqref{eq-Rnu1}, \eqref{eq-snR} with \eqref{eq-shR}, we have
\begin{align}
     N_{11}=&H\frac{u\phi''}{\phi\phi'}-2n(\frac{u\phi''}{\phi\phi'})^2+\frac{uH}{\phi(\phi')^2}(\phi\phi'''-\phi'\phi'')(1-\frac{u^2}{\phi^2}-\langle\partial_r,e_1\rangle^2)\nonumber\\
       &+2u\langle\nabla\frac{u}{\phi'},\nabla\frac{\phi''}{\phi\phi'}\rangle-n\frac{\phi''}{\phi}+\frac{u^2}{\phi'}\langle\phi\partial_r,\nabla\frac{\phi\phi'\phi'''-\phi(\phi'')^2-(\phi')^2\phi''}{\phi^3(\phi')^2}\rangle \nonumber\\
  &+u^2\frac{\phi}{\phi'}\left(5\frac{\phi'\phi''}{\phi^4}-\frac{\phi'''}{\phi^3}-4\frac{\phi'((\phi')^2-1)}{\phi^5}\right)\left(1-\frac{u^2}{\phi^2}-n\langle\partial_r,e_1\rangle^2\right)\nonumber\\
  &+2\frac{(\phi')^2-\phi\phi''-1}{\phi^2}trS\langle\partial_r,e_1\rangle^2+2u\frac{(\phi')^2-\phi\phi''-1}{\phi^2\phi'}\sum_{k=1}^n{S_{kk}\langle\partial_r,e_{k}\rangle^2}\nonumber\\
  &-(n+\frac{uH}{\phi'})\left[-\frac{(\phi')^2-1}{\phi^2}+\frac{(\phi')^2-\phi\phi''-1}{\phi^2}(\langle\partial_r,e_1\rangle^2+\frac{u^2}{\phi^2})\right]\nonumber\\
  &-2u\frac{(\phi')^2-1}{\phi^2\phi'}trS+2\frac{u^2\phi''}{(\phi')^2}\langle\partial_r,\nabla\frac{\phi''}{\phi\phi'}\rangle.\label{expN-4}
\end{align}
By the equation (\ref{eq-sTa}), we have
\begin{align}\label{expN-1}
	\phi\phi'\phi'''-(\phi')^2\phi''=(n-1)\frac{(\phi')^2-\phi\phi''-1}{\phi}(\phi')^2
\end{align}
and
\begin{align}\label{expN-2}
	-\frac{\phi'''}{\phi^2}+5\frac{\phi'\phi''}{\phi^3}+4\frac{\phi'(1-(\phi')^2)}{\phi^4}=-(n+3)\frac{(\phi')^2-\phi\phi''-1}{\phi^4}\phi'.
\end{align}
Then
\begin{align}\label{expN-3}
    \nabla\frac{\phi''}{\phi\phi'}=&\frac{\phi\phi'\phi'''-\phi(\phi'')^2-(\phi')^2\phi''}{\phi^2(\phi')^2}\sum_{k=1}^n\langle\partial_r,e_k\rangle e_k\nonumber\\
    =&(n-1)\frac{(\phi')^2-\phi\phi''-1}{\phi^3}\sum_{k=1}^n\langle\partial_r,e_k\rangle e_k-\frac{(\phi'')^2}{\phi(\phi')^2}\sum_{k=1}^n\langle\partial_r,e_k\rangle e_k.
\end{align}
Moreover, we have
\begin{align}\label{expN-5}
    &\nabla\frac{\phi\phi'\phi'''-\phi(\phi'')^2-(\phi')^2\phi''}{\phi^3(\phi')^2}\nonumber\\=&(n-1)\nabla\frac{(\phi')^2-\phi\phi''-1}{\phi^4}-2\frac{\phi''}{\phi\phi'}\nabla\frac{\phi''}{\phi\phi'}\nonumber\\
	=&-(n-1)(n+3)\frac{(\phi')^2-\phi\phi''-1}{\phi^5}\phi'\sum_{k=1}^n\langle\partial_r,e_k\rangle e_k\nonumber\\
	&-2(n-1)\phi''\frac{(\phi')^2-\phi\phi''-1}{\phi^4\phi'}\sum_{k=1}^n\langle\partial_r,e_k\rangle e_k+2\frac{(\phi'')^3}{\phi^2(\phi')^3}\sum_{k=1}^n\langle\partial_r,e_k\rangle e_k.
\end{align}
Now, substituting (\ref{preconv-1}), (\ref{expN-1}), (\ref{expN-2}), (\ref{expN-3}) and (\ref{expN-5}) into (\ref{expN-4}),  we finally get equation \eqref{eq-N11} by uniting similar terms.
\end{proof}
In order to prove that static convexity is preserved along the flow \eqref{flow-WVPF}, it's sufficient to show that
\begin{align*}
    N_{11}+\sup_{\Lambda} 2a^{k\ell}(2\Lambda_k^p\nabla_{\ell}S_{ip}v^i-\Lambda_k^p\Lambda_{\ell}^qS_{pq})\geq 0.
\end{align*}
Since $0=\nabla_kS_{11}=\nabla_k(h_{11}-\frac{u\phi''}{\phi\phi'})$ at the point $(x_0,t_0)$, we deduce that
\begin{align}
	&\sup_{\Lambda} 2a^{k\ell}(2\Lambda_k^p\nabla_{\ell}S_{ip}v^i-\Lambda_k^p\Lambda_{\ell}^qS_{pq})\notag\\
	=&2\frac{u}{\phi'}\sup_{\Lambda}\Big[\sum_{k,p=1}^n(2\Lambda_k^p\nabla_k S_{1p}-(\Lambda_k^p)^2 S_{pp})\Big]\notag\\
	=&2\frac{u}{\phi'}\sup_{\Lambda}\Big[\sum_{k=1}^n\sum_{p=2}^n(2\Lambda_k^p\nabla_k S_{1p}-(\Lambda_k^p)^2 S_{pp})\Big]\notag\\
	=&-2\frac{u}{\phi'}\sup_{\Lambda}\sum_{k=1}^n\sum_{p=2}^n S_{pp}\Big[\big(\Lambda_k^p-\frac{\nabla_k{S_{1p}}}{S_{pp}}\big)^2-\big(\frac{\nabla_k{S_{1p}}}{S_{pp}}\big)^2\Big]\notag\\
	=&2\frac{u}{\phi'}\sum_{k=1}^n\sum_{p=2}^n\frac{(\nabla_k{S_{1p}})^2}{S_{pp}}\geq2\frac{u}{\phi'}\sum_{p=2}^n\frac{(\nabla_1{S_{1p}})^2}{S_{pp}}.\label{eq-refined}
\end{align}
By Codazzi equation and equation (\ref{eq-gnR}), for any $p>1$ we have
\begin{align*}
	\nabla_1S_{1p}=&\nabla_1h_{1p}=\nabla_ph_{11}+\bar{R}_{\nu 1p1}\\
	=&\nabla_p(\frac{u\phi''}{\phi\phi'})+u\frac{(\phi')^2-\phi\phi''-1}{\phi^3}\langle\partial_r,e_p\rangle\\
	=&\frac{\phi''}{\phi'}S_{pp}\langle\partial_r,e_p\rangle+nu\frac{(\phi')^2-\phi\phi''-1}{\phi^3}\langle\partial_r,e_p\rangle.
\end{align*}
Hence,
\begin{align}\label{preconv-2}
	2\frac{u}{\phi'}\sum_{p=2}^n\frac{(\nabla_1{S_{1p}})^2}{S_{pp}}=&2u\frac{(\phi'')^2}{(\phi')^3}\sum_{p=2}^nS_{pp}\langle\partial_r,e_p\rangle^2+4n u^2\phi''\frac{(\phi')^2-\phi\phi''-1}{\phi^3(\phi')^2}(1-\frac{u^2}{\phi^2}-\langle \partial_r, e_1\rangle^2)\nonumber\\
	&+2n^2 u^3\frac{\left((\phi')^2-\phi\phi''-1\right)^2}{\phi^6\phi'}\sum_{p=2}^n\frac{\langle\partial_r,e_p\rangle^2}{S_{pp}}.
\end{align}
Combining (\ref{eq-N11}) with (\ref{preconv-2}), we have 
\begin{align*}
&N_{11}+2\frac{u}{\phi'}\sum_{p=2}^n\frac{(\nabla_1{S_{1p}})^2}{S_{pp}}\\
 =&n\frac{(\phi')^2-\phi\phi''-1}{\phi^2}\Bigg[\left(1+\frac{u}{\phi'}trS+(n+4)\frac{u^2\phi''}{\phi(\phi')^2}-(n+3)\frac{u^2}{\phi^2}\right)(1-\frac{u^2}{\phi^2}-\langle \partial_r, e_1\rangle^2)\nonumber\\
	&-\frac{2u}{n\phi'}trS(1-\langle \partial_r, e_1\rangle^2)
    +2\frac{u}{\phi'}\sum_{k=1}^nS_{kk}\langle\partial_r,e_k\rangle^2\Bigg]\\
	&+2n^2 u^3\frac{\left((\phi')^2-\phi\phi''-1\right)^2}{\phi^6\phi'}\sum_{p=2}^n\frac{\langle\partial_r,e_p\rangle^2}{S_{pp}}.\nonumber
\end{align*}
By the lower bound condition $(\phi')^2-\phi\phi''\geq0$, we see that
\begin{align*}
    (n+4)\frac{u^2\phi''}{\phi'^2\phi}\leq (n+4)\frac{u^2}{\phi^2}.
\end{align*}
Moreover, notice that $(\phi')^2-\phi\phi''-1\leq0$, $1-\frac{u^2}{\phi^2}-\langle \partial_r, e_1\rangle^2\geq0$ and $S_{pp}\leq trS$ for any $p$, we have
\begin{align}\label{preconv-3}
    &N_{11}+2\frac{u}{\phi'}\sum_{p=2}^n\frac{(\nabla_1{S_{1p}})^2}{S_{pp}}\nonumber\\
    \geq&n\frac{(\phi')^2-\phi\phi''-1}{\phi^2}\Bigg[\left(1+3\frac{u}{\phi'}trS+\frac{u^2}{\phi^2}\right)(1-\frac{u^2}{\phi^2}-\langle \partial_r, e_1\rangle^2)\nonumber\\
	&-\frac{2u}{n\phi'}trS(1-\langle \partial_r, e_1\rangle^2)
    +2n u^3\frac{(\phi')^2-\phi\phi''-1}{\phi^4\phi'}\frac{1-\frac{u^2}{\phi^2}-\langle \partial_r, e_1\rangle^2}{trS}\Bigg].
\end{align}
If $1-\frac{u^2}{\phi^2}-\langle\partial_r,e_1\rangle^2=0$ at $(x_0,t_0)$, we have
\begin{align*}
  N_{11}+2\frac{u}{\phi'}\sum_{p=2}^n\frac{(\nabla_1{S_{1p}})^2}{S_{pp}}
    \geq -2u\frac{(\phi')^2-\phi\phi''-1}{\phi^2\phi'}trS(1-\langle \partial_r, e_1\rangle^2)\geq0,
\end{align*}
which implies that the static convexity is preserved along the flow (\ref{flow-WVPF}). Otherwise, denote 
$$I:=\frac{1-\langle \partial_r, e_1\rangle^2}{1-\frac{u^2}{\phi^2}-\langle \partial_r, e_1\rangle^2},$$ 
then we have
\begin{align}\label{preconv-4}
    &N_{11}+2\frac{u}{\phi'}\sum_{p=2}^n\frac{(\nabla_1{S_{1p}})^2}{S_{pp}}\nonumber\\
    \geq&n\frac{(\phi')^2-\phi\phi''-1}{\phi^2}\Bigg[2+\frac{u}{\phi'}trS(3-\frac{2}{n}I)
    +2n\frac{(\phi')^2-\phi\phi''-1}{\phi^4\phi'}\frac{u^3}{trS}\Bigg](1-\frac{u^2}{\phi^2}-\langle \partial_r, e_1\rangle^2).
\end{align}
If $M_0$ is $\epsilon_0$-close to a slice of $\mathbf{N}^{n+1}$ in the $C^1$ sense, then by the definition \eqref{epclose} and Proposition \ref{prop-C0}, we know that the flow hypersurface $M_t$ remains to be $\epsilon_0$-close to a slice of $\mathbf{N}^{n+1}$ in the $C^1$ sense. Recall the condition (\ref{epclose}) says that $\frac{u^2}{\phi^2}\geq\frac{1}{1+\varepsilon_0}$. Therefore,
\begin{align}\label{In-lowerbound I}
  I\geq\frac{u^2}{\phi^2(1-\frac{u^2}{\phi^2})}\geq\frac{1}{\varepsilon_0}.
\end{align}
First we assume that $\varepsilon_0<\frac{2}{3n}$, then the Cauchy-Schwarz inequality leads to
\begin{align}\label{preconv-5}
    \frac{u}{\phi'}trS(3-\frac{2}{n}I)
    +2n\frac{(\phi')^2-\phi\phi''-1}{\phi^4\phi'}\frac{u^3}{trS}\leq-\frac{2 u^2}{\phi^2\phi'}\sqrt{(2I-3n)(1+\phi\phi''-(\phi')^2)}.
\end{align}
By \eqref{eq-stat-2}, there exists an universal constant $C_0\geq0$ such that
\begin{align}\label{preconv-6}
    (\phi')^2-\phi\phi''-1=-\frac{C_0}{\phi^{n-1}},\ \forall r\geq r_0.
\end{align}
We only consider the case that $C_0>0$, otherwise $\mathbf{N}^{n+1}$ is of constant sectional curvature, Hu and Li \cite{HL-2019} proved the preserving of the static convexity in this case without any further assumption. Thus, substituting (\ref{preconv-6}) into (\ref{preconv-5}) we get
\begin{align}\label{preconv-7}
   &\frac{u}{\phi'}trS(3-\frac{2}{n}I)
    +2n\frac{(\phi')^2-\phi\phi''-1}{\phi^4\phi'}\frac{u^3}{trS}\nonumber\\
    \leq& -2\sqrt{C_0(2I-3n)}\frac{u^2}{\phi^{\frac{n+3}{2}}\phi'}\nonumber\\
    \leq&-\frac{2}{1+\varepsilon_0}\sqrt{C_0(\frac{2}{\varepsilon_0}-3n)}\frac{1}{\phi^{\frac{n-1}{2}}\phi'},
\end{align}
where we have used \eqref{epclose} and \eqref{In-lowerbound I} in the second inequality. Since $M_t$ lie in the bounded domain $B(R)$, $\phi$ and $\phi'$ are monotone increasing functions of $r$, then the constant $\varepsilon_0$ can be chosen depending only on $n$ and $R$
such that 
\begin{align*}
   -\frac{2}{1+\varepsilon_0}\sqrt{C_0(\frac{2}{\varepsilon_0}-3n)}\frac{1}{\phi'\phi^{\frac{n-1}{2}}}\leq -2. 
\end{align*}
Therefore, combining (\ref{preconv-7}) with (\ref{preconv-4}), we obtain
\begin{align*}
    N_{11}+2\frac{u}{\phi'}\sum_{p=2}^n\frac{(\nabla_1{S_{1p}})^2}{S_{pp}}\geq0.
\end{align*}
By the tensor maximum principle, we have proved that the static convexity is preserved along the flow (\ref{flow-WVPF}).

Finally, we show that $M_t$ becomes strictly static convex for $t>0$ unless $M_{t}\equiv S(r_0)$ with $((\phi')^2-\phi\phi'')(r_0)=0$. By the strong maximum principle, the inequality becomes strict unless $\kappa_1\equiv\frac{u\phi''}{\phi\phi'}$ on $M_0$ everywhere. By Proposition \ref{sconvpoint},  there exists at least one point where all the principal curvatures are strictly greater than $\frac{u\phi''}{\phi\phi'}$ unless $M_{0}=S(r_0)$ with $((\phi')^2-\phi\phi'')(r_0)=0$ (hence $M_t\equiv S(r_0)$ since a slice is a static solution along the flow \eqref{flow-WVPF}). This completes the proof of Theorem \ref{preservingofconvex}.

\section{Applications to the weighted geometric inequalities}\label{sec-inequality}
Let $\mathbf{N}^{n+1}$ be a static, rotationally symmetric space equipped with the metric \eqref{eq-metric} and satisfies the assumptions in Definition \ref{defn-N}.  In this section,  we complete the proof of Corollary \ref{Cor-wii}. We first show the following monotonicities of $A_{i, \phi}(M_t)$ $(i=0,1)$ along the flow \eqref{flow-WVPF}.

\begin{prop}\label{Prop-monto}
Let $\mathbf{N}^{n+1}$ be a static, rotationally symmetric space and satisfies the assumptions in Definition \ref{defn-N} and $M_t$ be a smooth, static convex and graphical solution of the flow \eqref{flow-WVPF}. Then along the flow \eqref{flow-WVPF}, the weighted area $A_{0,\phi}(M_t)$ is monotone decreasing if $n\geq 2$ and the weighted mean curvature integral $A_{1,\phi}(M_t)$ is monotone decreasing if $n\geq 3$.
\end{prop}
\begin{rem}
    In fact, the inequality  $(\phi')^2-\phi\phi''\geq 0$ in \eqref{in-phiphi} is no need for deducing the  monotonicities of $A_{i,\phi}(M_t)$ $(i=0,1)$.
\end{rem}

\begin{proof}
    First, if $n\geq 2$, using the evolution equation (\ref{va-Aphi}), we have
   \begin{align}
   	\frac{d}{dt}A_{0,\phi}(M_t)=&\int_{M_t}\left[(n-\frac{uH}{\phi'})\frac{\phi''}{\phi}u+\phi'(n-\frac{uH}{\phi'})H\right]\,d\mu\notag\\
   	=&\int_{M_t}\left[\frac{2n}{n-1}\sigma_2-H^2+\frac{n}{n-1}\left(\overline{Ric}(\partial_r,\partial_r)-\overline{Ric}(\nu,\nu)\right)\right]u\,d\mu,\notag\\
    &+\int_{M_t}\left[{(n-\frac{uH}{\phi'})\frac{\phi''}{\phi}u}\right]\,d\mu\nonumber\\
   	\leq&\int_{M_t}\left[{(n-\frac{uH}{\phi'})\frac{\phi''}{\phi}u}\right]\,d\mu+\frac{n}{n-1}\int_{M_t}\left(\overline{Ric}(\partial_r,\partial_r)-\overline{Ric}(\nu,\nu)\right)u\,d\mu\label{ineq-WA}
   \end{align}
where in the second equality we have used the Minkowski identity \eqref{eq-Min2} and Newton-MacLaurin inequality in the last line of \eqref{ineq-WA}. By a direct calculation, we have
\begin{align}\label{eq-LA}
	\Delta_g\phi'=&\frac{\phi\phi'''-\phi'\phi''}{\phi^3}(\phi^2-u^2)+\frac{\phi''}{\phi}(n\phi'-uH)\nonumber\\
 =&(n-1)\frac{(\phi')^2-\phi\phi''-1}{\phi^2}(1-\frac{u^2}{\phi^2})+\frac{\phi''}{\phi}(n\phi'-uH),
\end{align}
where in the second equality we used the static condition (\ref{eq-subst}).
Combining inequality \eqref{ineq-WA} with equation \eqref{eq-LA}, we have
\begin{align}\label{Ineq-A'}
	\frac{d}{dt}A_{0,\phi}(M_t)\leq&\int_{M_t}\left[\frac{n}{n-1}\left(\overline{Ric}(\partial_r,\partial_r)-\overline{Ric}(\nu,\nu)\right)-(n-1)\frac{(\phi')^2-\phi\phi''-1}{\phi^2}(1-\frac{u^2}{\phi^2})\right]u\,d\mu\nonumber\\
 &+\int_{M_t}{\frac{u}{\phi'}\Delta_g\phi'}\,d\mu.
\end{align}
Since 
\begin{align}
	\int_{M_t}{\frac{u}{\phi'}\Delta_g\phi'}\,d\mu&=-\int_{M_t}{\langle\nabla{\frac{u}{\phi'},\nabla{\phi'}}\rangle}\,d\mu\notag\\
	&=-\int_{M_t}{\frac{\phi\phi''}{\phi'}(h^{st}-\frac{u\phi''}{\phi\phi'}g^{st})\langle\partial_{r},e_s\rangle\langle\partial_{r},e_t\rangle}\,d\mu\notag\\
	&\leq 0.\label{In-eq-L}
\end{align}
The inequality holds due to that $\phi''>0$ and $M_t$ is a static convex solution of the flow \eqref{flow-WVPF}.

By \eqref{eq-RicciT}, we have
\begin{align}
    \overline{Ric}(\partial_r,\partial_r)&=-n\frac{\phi''}{\phi},\label{eq-Ricrr}\\
    \overline{Ric}(\nu,\nu)&=-\frac{\phi\phi''+(n-1)((\phi')^2-1)}{\phi^2}+(n-1)u^2\frac{(\phi')^2-\phi\phi''-1}{\phi^4}\label{eq-Ricnunu}
\end{align}
and hence
\begin{align}\label{in-dr}
	\overline{Ric}(\partial_r,\partial_r)-\overline{Ric}(\nu,\nu)=(n-1)\frac{(\phi')^2-\phi\phi''-1}{\phi^2}(1-\frac{u^2}{\phi^2}).
\end{align}
Then combining \eqref{Ineq-A'}, \eqref{In-eq-L} with \eqref{in-dr}, we have
\begin{align*}
   \frac{d}{dt}A_{0,\phi}(M_t)\leq&\int_Mu\frac{(\phi')^2-\phi\phi''-1}{\phi^2}(1-\frac{u^2}{\phi^2})d\mu.
\end{align*}
Now, we conclude that $A_{0,\phi}(M_t)$ is monotone decreasing under the assumption \eqref{in-phiphi}.

Next, we prove that $A_{1,\phi}(M_t)$ is monotone decreasing along the flow if $n\geq 3$. Using the evolution equation \eqref{va-Hphi}, we have 
\begin{align}\label{Ineq-A2-eq1}
    \partial_tA_{1,\phi}(M_t)=&\int_M(n-\frac{uH}{\phi'})\frac{\phi''}{\phi}uH-\left(\Delta_g\phi'-2\phi'\sigma_2+\phi'\overline{Ric}(\nu,\nu)\right)(n-\frac{uH}{\phi'})d\mu.
\end{align}
Combining (\ref{eq-LA}), (\ref{in-dr}) with (\ref{Ineq-A2-eq1}), we see that
\begin{align}\label{Ineq-A2-eq2}
   \partial_tA_{1,\phi}(M_t)
   =&2\int_{M}\frac{\phi''u}{\phi}(nH-\frac{u}{\phi'}H^2)d\mu+2\int_{M}n\phi'\sigma_2-uH\sigma_2d\mu\nonumber\\
	\leq&2\int_{M}\frac{\phi''u}{\phi}(nH-\frac{2n}{n-1}\frac{u}{\phi'}\sigma_2)d\mu+2\int_{M}n\phi'\sigma_2-\frac{3n}{n-2}u\sigma_3d\mu\nonumber\\
 =&2\int_{M}\frac{\phi''u}{\phi}(nH-\frac{2n}{n-1}\frac{u}{\phi'}\sigma_2)d\mu+2\frac{n}{n-1}\int_M\sigma_2^{ij}\overline{Ric}(e_i,\nu)\langle\phi\partial_r,e_j\rangle.
\end{align}
where we used Newton-Maclaurin inequalities $H^2\geq \frac{2n}{n-1}\sigma_2$, $H\sigma_2\geq\frac{3n}{n-2}\sigma_3$ and Minkowski formula (\ref{eq-Min3}).
Then, by a direct calculation, we have 
\begin{align}\label{Ineq-A2-eq3}
	\sigma_2^{ij}\nabla_i\nabla_j\phi'=&\sigma_2^{ij}(\nabla_i(\frac{\phi''}{\phi}\langle\phi\partial_r,e_j\rangle))\nonumber\\
	=&\sigma_2^{ij}\left(\frac{\phi''}{\phi}(\phi'g_{ij}-uh_{ij})+\frac{\phi\phi'''-\phi'\phi''}{\phi}\langle\partial_r,e_i\rangle\langle\partial_r,e_j\rangle\right)\nonumber\\
	=&\frac{\phi''}{\phi}((n-1)\phi'H-2u\sigma_2)+(n-1)\phi'\frac{(\phi')^2-\phi\phi''-1}{\phi^2}\sigma_2^{ij}\langle\partial_r,e_i\rangle\langle\partial_r,e_j\rangle,
\end{align}
where in the last equality we used the static condition (\ref{eq-subst}). 
In the meanwhile, by (\ref{eq-RicciT}) we notice that
\begin{align}\label{Ineq-A2-eq4}
   \overline{Ric}(e_i,\nu)=&(n-1)u\frac{(\phi')^2-\phi\phi''-1}{\phi^3}\langle\partial_r,e_i\rangle.
\end{align}
Substituting (\ref{Ineq-A2-eq3}) and (\ref{Ineq-A2-eq4}) into (\ref{Ineq-A2-eq2}), we get
\begin{align*}
   \partial_tA_{1,\phi}(M_t)\leq&\frac{2n}{n-1}\int_{M}\frac{u}{\phi'}\sigma_2^{ij}\nabla_i\nabla_j\phi'\\
   =&-\frac{2n}{n-1}\int_M\sigma_2^{ij}\nabla_i\frac{u}{\phi'}\nabla_j\phi'd\mu
	-\frac{2n}{n-1}\int_{M}\frac{u}{\phi'}(\sigma_2)^{ij}_{i}\nabla_j\phi'\\
	=&-\frac{2n}{n-1}\int_M\sigma_2^{ij}\nabla_i\frac{u}{\phi'}\nabla_j\phi'd\mu+\frac{2n}{n-1}\int_{M}\frac{u\phi''}{\phi'}\overline{Ric}(e_i,\nu)\langle\partial_r,e_i\rangle\\
	=&-\frac{2n}{n-1}\int_M\frac{\phi\phi''}{\phi'}\sigma_2^{ii}S_{ii}\langle\partial_r,e_i\rangle^2 d\mu+2n\int_{M}u^2\phi''\frac{(\phi')^2-\phi\phi''-1}{\phi^3\phi'}(1-\frac{u^2}{\phi^2})d\mu,
\end{align*}
where we used the Minkowski identity (\ref{eq-Min3}), equations (\ref{preconv-1}) and (\ref{Ineq-A2-eq4}). Since $M_t$ is a static convex solution of the flow \eqref{flow-WVPF}, we conclude  that $\sigma_2^{ij}=Hg^{ij}-h^{ij}$ is positive definite. Then, 
\[\partial_tA_{1,\phi}(M_t)\leq 0\]
by the assumption \eqref{in-phiphi}.
\end{proof}
Finally, we give the proof of Corollary \ref{Cor-wii}.
\begin{proof}[Proof of Corollary \ref{Cor-wii}]
First, we claim that if $\mathbf{N}^{n+1}$ is a static rotationally symmetric space under the assumptions in Definition \ref{defn-N}, then the single variable functions $\chi_{i,\alpha}(x)$ defined in \eqref{eq-chi} are strictly increasing functions of $x$. To see this, we only need to check that $A_{i,\phi}(r)$ are both strictly increasing functions of $r$ for $i=0$ and $1$, since $V_{\phi}^{\alpha}(r)$ are strictly increasing functions of $r$. By definitions of $A_{i,\phi}(r)(i=0,1)$ and the graphical representation of a hypersurface in $\mathbf{N}^{n+1}$ in \S\ref{subsection-graph}, we have
\begin{align*}
A_{0,\phi}(r)&=\omega_n\phi^n\phi',\\
A_{1,\phi}(r)&=n\omega_n\phi^{n-1}(\phi')^2,
\end{align*}
where $\omega_n$ is the area of the $n$-dimensional unit sphere, and hence
\begin{align}
    \frac{d}{dr}A_{0,\phi}(r)&=\omega_n\phi^{n-1}\left[n(\phi')^2+\phi\phi''\right]>0,\label{eq-dA0}\\
    \frac{d}{dr}A_{1,\phi}(r)&=n\omega_n\phi^{n-2}\phi'\left[(n-1)(\phi')^2+2\phi\phi''\right]>0,
\end{align}
since $\phi''>0$ by assumptions.

Then the inequalities \eqref{In-wii} and \eqref{In-wiii} follow from Theorem \ref{Thm-main}, Theorem \ref{preservingofconvex}, Lemma \ref{lem-mono} and Proposition \ref{Prop-monto} directly. In precise, if the flow \eqref{flow-WVPF} starting from $M$ converges to a slice $S(r_{\infty})$, then we have
\begin{equation}
    A_{i,\phi}(M)\geq A_{i,\phi}(r_{\infty})=\chi_{i,\alpha}(V_{\phi}^{\alpha}(r_{\infty}))\geq \chi_{i,\alpha}(V_{\phi}^{\alpha}(\Omega))
\end{equation}
for both $i=0,1$.

Finally, we deal with the equality case. According to Theorem \ref{preservingofconvex}, the flow hypersurfaces $M_t$ becomes strictly static convex unless $M=S(r_0)$ with $((\phi')^2-\phi\phi'')(r_0)=0$. If $M=S(r_0)$ with $((\phi')^2-\phi\phi'')(r_0)=0$, we have nothing to prove. Hence without loss of generality, we assume that $M_t$ becomes strictly static convex with initial hypersurface $M$. Then we return back to check the proof of Proposition \ref{Prop-monto}, we can easily see that $A_{i,\phi}(M_t)$ $(i=0,1)$ is strictly monotone decreasing for $t>0$ unless $M_t$ is a slice of $\mathbf{N}^{n+1}$. Since a slice is a static solution of the flow $\eqref{flow-WVPF}$, this implies that $M$ is also a slice of $\mathbf{N}^{n+1}$. This completes the proof of Corollary \ref{Cor-wii}.
\end{proof}

\end{document}